\documentclass[a4paper,12pt]{amsart}

\usepackage{graphicx,xcolor}

\usepackage{amsmath, amssymb}
\usepackage{amsthm}
\usepackage{lineno}
\theoremstyle{plain}
\newtheorem{theorem}{Theorem}
\newtheorem{lemma}{Lemma}
\newtheorem{proposition}{Proposition}
\newtheorem{corollary}{Corollary}

\theoremstyle{definition}

\theoremstyle{remark}
\newtheorem*{remark}{Remark}

 \def\S{\mathbb{S}}
 \def\R{\mathbb{R}}
  \def\H{\mathbb{H}}
 \def\C{\mathbb{C}}
 \def\D{\mathcal{D}}

\title[CR Submanifolds]{CR submanifolds of the nearly K\" ahler $\S^3\times\S^3$ characterised by properties of the almost product structure}

\author{Miroslava Anti\' c}
\address{  University of Belgrade, Faculty of Mathematics, Studentski trg 16, 11000 Belgrade, Serbia}
\email{mira@matf.bg.ac.rs}

\author{Nata\v sa Djurdjevi\' c }
\address{ University of Belgrade, Faculty of Mathematics, Studentski trg 16, 11000 Belgrade, Serbia}
\email{natasa@matf.bg.ac.rs}

\author[Moruz]{Marilena Moruz}
\address{KU Leuven, Department of Mathematics, Celestijnenlaan 200B -- BOX 2400, 3001 Leuven, Belgium}
\email{ marilena.moruz@kuleuven.be}

\begin{document}

\begin{abstract}
In a previous paper \cite{ADjMV}, the authors together with L. Vrancken initiated the study of $3$-dimensional CR submanifolds of the nearly K\" ahler homogeneous $\S^3\times \S^3$. As is shown by Butruille this is one of only four homogeneous $6$-dimensional nearly K\"ahler manifolds. Besides its almost complex structure $J$ it also admits a canonical almost product structure $P$, see \cite{MV} and \cite{BDDV}. Along a $3$-dimensional CR submanifold the tangent space of $\S^3\times\S^3$ can be naturally split as the orthogonal sum of three $2$-dimensional vector bundles $\mathcal D_1$, $\mathcal D_2$ and $\mathcal D_3$.  We study the CR submanifolds in relation to the behavior of the almost product structure on these vector bundles.
\end{abstract}



\maketitle

\section{Introduction}

An almost Hermitian manifold $(\widetilde{M}, g, J)$, with Levi-Civita connection
$\widetilde{\nabla}$, is called a nearly K\" ahler manifold if for any tangent vector
$X$ it holds
$(\widetilde{\nabla}_X J)X=0$. If, moreover, $\widetilde{\nabla}J$ is a vanishing
tensor,  $\widetilde{M}$ is said to be a K\" ahler manifold. For that reason the skew symmetric tensor $G(X,Y) = (\widetilde{\nabla}_X J)Y$ plays an important role in the nearly K\"ahler geometry. A nearly K\"ahler manifold is called strict if and only if for any non zero vector $X$, $G(X,Y)$ does not vanish identically.
 It is known that there exist only four $6$-dimensional homogeneous nearly K\" ahler
 manifolds, that are not K\" ahler: the sphere $\S^6$, the complex projective space
 $\C P^3$, the
 flag manifold $\mathbb F^3$ and $\S^3\times\S^3$, see \cite{B}. One should also remark that the first examples of complete non homegeneous K\"ahler manifolds were recently discovered by Foscolo and Haskins in \cite{Haskins}.

In this paper we are in particular interested in the study of $3$-dimensional CR submanifolds of $\S^3\times\S^3$.
In general, a submanifold $M$ is called a CR submanifold
if there exists a $C^\infty$-differential $J$ invariant
distribution $\D_1$ on
$M$ (i.e., $J\D_1=\D_1$), such that its orthogonal
complement $\D_1^\perp$ in $TM$ is totally real
($J\D_1^\perp \subseteq T^\bot M$), where $T^\bot M$ is the
normal bundle over $M$. We say that $M$ is proper if it is neither almost complex, nor
totally
real. Note that, in the specific case of a $3$-dimensional submanifold $M$ of a $6$-dimensional (nearly) K\"ahler manifold, we have that $M$ is a proper CR submanifold if and only if $JT_pM \cap T_pM$ is a $2$-dimensional integrable distribution.

Along a $3$-dimensional CR submanifold $M$ of a $6$-dimensional nearly K\"ahler manifold $\tilde M$ the tangent space of $\tilde M$ can be naturally split as the orthogonal sum of three $2$-dimensional $J$ invariant vector bundles $\mathcal D_1$, $\mathcal D_2$ and $\mathcal D_3$ over $M$. Here $\mathcal D_1$ is as defined before, i.e. $\mathcal D_1 =JT_pM \cap T_pM$, $\mathcal D_2= \mathcal D_1^\perp \oplus J \mathcal D_1^\perp$ and $\mathcal D_3= G(\mathcal D_1,\mathcal D_2)$.

In contrast to the nearly K\"ahler manifold $\mathbb{S}^6$ whose CR submanifolds have been extensively studied in for example \cite{HM}, \cite{HMS}, \cite{S}; the manifold $\S^3\times\S^3$ admits an additional almost product structure $P$. Here we
 study the CR submanifolds in relation to the behavior of the almost product structure on these vector bundles. In doing so we obtain the first examples, together with their characterisations, of CR submanifolds for which $\mathcal D_1$ is not an integrable distribution. Note that examples with integrable $\mathcal D_1$ distribution can be easily obtained by moving an almost complex surface along by a $1$-parameter family of isometries.

\section{Preliminaries}
\label{prelim}

In the usual way we consider $\S^3$ as the unit sphere in $\R^4$ identified with the quaternions $\H$.
 Therefore, we can represent an arbitrary tangent vector at a point $(p,q)\in \S^3\times
 \S^3$ by $Z=(p\alpha, q\beta)$, where $\alpha$ and $\beta$ are imaginary
 quaternions.

The almost complex structure on $\S^3\times \S^3$ is given by, see for example \cite{BDDV,B}:
\begin{align*}
JZ_{(p, q)}=\frac{1}{\sqrt{3}}( p (2 \beta-\alpha), q(-2 \alpha +\beta))
\end{align*}
and a  compatible metric $g$  can be defined by
\begin{align*}
g(Z, Z')=\frac{1}{2}(\langle Z, Z'\rangle+\langle JZ, JZ'\rangle).
\end{align*}
where $\langle ., .\rangle$ is the standard metric on the quaternions.

It is known that this is indeed a nearly K\"ahler manifold and that therefore  the skew symmetric tensor field $G(X, Y)=(\widetilde{\nabla}_XJ)Y$ satisfies
 \begin{align}
&G(X, Y)+G(Y, X)=0,\nonumber\\
&G(X, JY)+JG(X, Y)=0,\nonumber\\
&g(G(X, Y),Z)+g(G(X, Z), Y)=0.\nonumber\\
&g(G(X, Y), G(Z, W))=\frac{1}{3}(g(X, Y)g(Y, W)-g(X, W)g(Y, Z)\nonumber\\
&\qquad\qquad \qquad +g(JX,Z)JY-g(JX, W)g(JZ, Y)),\label{Dzi1}
\end{align}

Also, in \cite{BDDV} the following almost product structure $P$ was introduced
\begin{align*}
&P(p\alpha, q \beta)=(p \beta,q \alpha),
\end{align*}
and it was also shown that it holds
\begin{align}\label{popste}
&P^2=Id,& &\hspace{-2mm}PJ=-JP,\\
&g(PZ, PZ')=g(Z, Z'),& &\hspace{-2mm}g(PZ, Z')=g(Z, PZ'),\nonumber\\
&PG(X, Z)+G(PX, PY)=0,& &\hspace{-2mm}(\tilde \nabla_X P)Y =\frac{1}{2} (J G(X, PY)+JPG(X, Y)).\nonumber
\end{align}
We note that the curvature tensor of $\widetilde{\nabla}$ is given by
\begin{align}\label{krivina}
&\widetilde{R}(X, Y)Z=\frac{5}{12}(g(Y,Z)X-g(X, Z)Y) \\
&+\frac{1}{12}(g(JY, Z)JX-g(JX,Z)JY-2g(JX,Y)JZ)\nonumber\\
&+\frac{1}{3}(g(PY,Z)PX-g(PX,Z)PY+
g(JPY,Z)JPX-g(JPX,Z)JPY).\nonumber
\end{align}

From \cite{MV} we see that the maps ${\mathcal F_{abc}}$, ${\mathcal F_1}$ and ${ \mathcal F_2}$ of $\mathbb{S}^3 \times   \mathbb{S}^3  $
defined respectively by
\begin{align*}
\mathcal F_{abc}(p,q)&=(a p\bar c,b q \bar c),\\
 \mathcal F_{1}(p,q) &=(q,p),\\
 \mathcal F_{2}(p,q) &=(\bar p,q \bar p),
\end{align*}
where $a,b,c$ are unitary quaternions are isometries (and therefore preserve the metric). The maps ${\mathcal F_{abc}}$ also preserve the almost complex structure $J$ and the almost product structure $P$. The same is not true for the group generated by the isometries $\mathcal F_1$ and $\mathcal F_2$. The isometries of this group still preserve $J$, at least up to sign. However they do not necessarely preserve $P$. This is due to the fact that, see \cite{MV}, there are actually three canonical almost product structure (satisfying the same fundamental properties), and the isometries of this group allow to switch between these different structures. In view of this it is clear that $\mathcal F_{abc}, \mathcal F_{1}, \mathcal F_{2}$ map a CR submanifold of $\S^3 \times \S^3$ to a CR submanifold of $\S^3 \times \S^3$. Moreover these maps also preserve the bundles $\mathcal D_1$ and $\mathcal D_1^\perp$.

Note that the structure  $P$ can be  expressed in terms of the usual product structure $Q(p\alpha,q \beta)=(-p\alpha,q \beta)$ by \begin{align}\label{Q}
      & QJ(Z)=\frac{1}{\sqrt{3}}(−2PZ + Z).
   \end{align}
Using this relation, for tangent vector $Z=(p\alpha,q \beta)=(U,V)$, we can obtain $U$ and $V$ in the following way
\begin{align}\label{Qp}
   & (U,0)=\frac{1}{2}(Z - Q(Z)),\hspace{15mm} (0,V)=\frac{1}{2}(Z + Q(Z)).
\end{align}
Also, the standard metric is given in terms of the nearly K\" ahler metric in the following way
\begin{align}
\langle Z, Z'\rangle=g(Z, Z')+\frac{1}{2}g(Z, PZ').\label{vezica}
\end{align}

Finally, for $X=(p\alpha,q\beta)$, $Y=(p\gamma,q\delta)$ $\in T_{(p,q)}\S^3\times \S^3$  it follows that
\begin{align}
G(X,Y )=&\frac{2}{3\sqrt{3}}(p(\beta \times \gamma + \alpha\times \delta + \alpha \times \gamma -2 \beta \times\delta),\nonumber\\
&q(-\alpha\times\delta - \beta\times \gamma + 2\alpha\times \gamma - \beta\times\delta)).\label{Gs}
\end{align}

Recall that for tangent vector fields $X, Y$ and normal $\xi$, the formulas of Gauss
and Weingarten are given by
\begin{align*}
&\widetilde{\nabla}_X Y=\nabla_X Y+h(X, Y),\\
&\widetilde{\nabla}_X \xi=-A_\xi X+\nabla^\perp_X\xi,
\end{align*}
where $\nabla, \nabla^\perp$ are, respectively,  the induced connection of the
submanifold and the connection in the normal bundle, while $h$ and $A$ are the  second
fundamental form
and the shape operator. Then we have that $g(h(X, Y), \xi)=g(A_\xi X, Y)$.

In \cite{DLMV} it was shown that the relation between the Euclidean connection
$\nabla^E$ of $\S^3\times\S^3$ and $\widetilde{\nabla}$ is given by
\begin{align}
\nabla^E_XY=\widetilde{\nabla}_X Y+\frac{1}{2}(JG(X, PY)+JG(Y, PX)).\label{veza1}
\end{align}

One should  notice the following. Since the connection $D$ in the space $\R^8$ satisfies
$D_{E_i}f=df(E_i)=(p\alpha_i, q\beta_i)$, we have that
\begin{align}
&D_{E_j}D_{E_i} f=D_{E_j}(p\alpha_i, q\beta_i)=(p\alpha_j\alpha_i+p E_j(\alpha_i),
q\beta_j\beta_i+q E_j(\beta_i))\nonumber\\
&=-(\langle\alpha_j,\alpha_i\rangle p, \langle \beta_j, \beta_i\rangle
q)+(p(\alpha_j\times\alpha_i+E_j(\alpha_i)),
q(\beta_j\times\beta_i+E_j(\beta_i))).\nonumber
\end{align}
Since the second fundamental form of the immersion of $\S^3\times\S^3$ is given by
\begin{align*}\widehat{h}(E_i, E_j)=-(\langle\alpha_j,\alpha_i\rangle p, \langle
\beta_j, \beta_i\rangle q),\end{align*}
we have that
\begin{align}
\nabla^E_{E_j}E_i=(p(\alpha_j\times\alpha_i+E_j(\alpha_i)),
q(\beta_j\times\beta_i+E_j(\beta_i))).\label{koneksija}
\end{align}

\section{The suitable moving frame for $3$-dimensional CR submanifolds}

We now recall from \cite{ADjMV} a moving frame along a $3$-dimensional proper CR submanifold
$M$ suitable for making computations. We have that the almost complex distribution $\D_1$ is
$2$-dimensional, while the totally real distribution $\D_1^\perp$ is of
dimension one. We can take unit vector fields $E_1$ and $E_2=JE_1$ that span $\D_1$,
and $E_3$ that spans $\D_1^\perp$.
We consider the  nearly K\" ahler metric to be denoted by $g$ throughout the paper, if it is not
explicitly stated otherwise. We have that $E_4=JE_3$ is a unit normal vector field. Then, from \cite{ADjMV} we know that the vector fields $E_1,\dots, E_4$, $E_5=\sqrt{3}G(E_1, E_3)\quad
\text{and}\quad E_6=\sqrt{3}G(E_2, E_3)=-JE_5$ form an orhtonormal moving frame.

In the same manner we obtain the following equalities
\begin{align}
&G(E_1, E_2)=0,& &G(E_1, E_3)=\frac{1}{\sqrt{3}}E_5,& &G(E_1, E_4)=\frac{1}{\sqrt{3}}
E_6,\nonumber\\
&G(E_1, E_5)=-\frac{1}{\sqrt{3}}E_3,& &G(E_1, E_6)=-\frac{1}{\sqrt{3}}E_4,& &G(E_2,
E_3)=\frac{1}{\sqrt{3}}E_6,\nonumber\\
&G(E_2, E_4)=-\frac{1}{\sqrt{3}}E_5,& &G(E_2, E_5)=\frac{1}{\sqrt{3}}E_4,& &G(E_2,
E_6)=-\frac{1}{\sqrt{3}}E_3,\nonumber\\
&G(E_3, E_4)=0,& &G(E_3, E_5)=\frac{1}{\sqrt{3}}E_1,& &G(E_3,
E_6)=\frac{1}{\sqrt{3}}E_2,\nonumber\\
&G(E_4, E_5)=-\frac{1}{\sqrt{3}}E_2,& &G(E_4, E_6)=\frac{1}{\sqrt{3}}E_1,& &G(E_5,
E_6)=0.\label{tabela}
\end{align}
Note that something similar can actually be done for a $3$-dimensional CR submanifold of any $6$-dimensional strict nearly K\"ahler manifold.\\
Under the assumption that $\{E_1, E_2, E_3\}$ is a positively oriented tangent frame of $M$, the vector fields $E_3$ and $E_4$ are uniquely determined. However, we do have a freedom
of rotating
$E_1$ in the almost complex distribution $\D_1$, which implies a corresponding rotation in $\mathcal D_3$. Then for a rotation angle $\varphi$ we have
\begin{align}
\widetilde{E}_1&=\cos\varphi E_1+\sin\varphi E_2,&  \widetilde{E}_2&=JE_1=-\sin\varphi
E_1+\cos\varphi E_2,\nonumber\\
\widetilde{E}_3&=E_3,& \widetilde{E}_4&=E_4,\nonumber\\
\widetilde{E}_5&=\cos\varphi E_5+\sin\varphi E_6,& \widetilde{E}_6&=-\sin
E_5+\cos\varphi E_6.\label{rot}
\end{align}

Now, let us denote the components of the connection, the second fundamental form and the normal connection respectively by
\begin{align*}
&\Gamma_{ij}^k=g(\widetilde{\nabla}_{E_i}E_j, E_k),\quad
h_{ij}^k=g(\widetilde{\nabla}_{E_i}E_j, E_{k+3}),\quad
b_{ij}^k=g(\widetilde{\nabla}_{E_i}E_{j+3}, E_{k+3}),
\end{align*}
for $1\leq i,j,k\leq 3$. Since the second fundamental form is symmetric, and
$\widetilde{\nabla}$ is the Levi-Civita we have that
\begin{align*}
\Gamma_{ij}^k=-\Gamma_{ik}^j,\quad b_{ij}^k=-b_{ik}^j,\quad
h_{ij}^k=h_{ji}^k.
\end{align*}

We recall from \cite{ADjMV} following three lemmas
\begin{lemma}\label{lemma1}
The coefficients $\Gamma_{ij}^k, h_{ij}^k, b_{ij}^k$ satisfy
\begin{align*}
&\Gamma_{11}^3=h_{12}^1,\hspace{4mm} \Gamma_{12}^3=-h_{11}^1,\hspace{4mm}
\Gamma_{21}^3=h_{22}^1, \hspace{4mm} \Gamma_{22}^3=-h_{12}^1,\hspace{4mm}
\Gamma_{31}^3=h_{23}^1,\\
&\Gamma_{32}^3=-h_{13}^1,\hspace{3mm}h_{11}^2=-h_{12}^3,\hspace{3mm}
h_{12}^2=h_{11}^3,\hspace{3mm}
h_{13}^3=h_{23}^2+\frac{1}{\sqrt{3}},\hspace{3mm}h_{22}^2=h_{12}^3,\\
&h_{22}^3=-h_{11}^3,\hspace{4mm} h_{23}^3=-h_{13}^2,\hspace{4mm}
b_{11}^2=h_{13}^3+\frac{1}{\sqrt{3}}, \hspace{4mm} b_{11}^3=-h_{13}^2,\\
&b_{21}^2=-h_{13}^2,\hspace{4mm}
b_{21}^3=-h_{13}^3+\frac{2}{\sqrt{3}},  \hspace{4mm}
b_{31}^2=h_{33}^3,  \hspace{4mm}
b_{31}^3=-h_{33}^2.
\end{align*}

\end{lemma}

\begin{lemma}It holds
\begin{align*}
&b_{12}^3=\Gamma_{11}^2-\Gamma_{32}^3,\hspace{3mm}
b_{22}^3=\Gamma_{21}^2+\Gamma_{31}^3, \hspace{3mm}
b_{32}^3=h_{33}^1+\Gamma_{31}^2
\end{align*}
\end{lemma}

The tensor field $P$ can be written as it follows.

\begin{lemma} There exists an open dense subset of $M$ such that, with respect to a suitable choice of $E_1$ belonging to  $\D_1$,  the tensor field
$P$ is given in the frame $E_1,\dots, E_6$ by
\begin{align}
PE_1&= \cos\theta E_1+a_1\sin\theta E_3+a_2\sin\theta E_4+a_3\sin\theta
E_5+a_4\sin\theta E_6,\nonumber\\
PE_2&=-\cos\theta E_2+a_2\sin\theta E_3-a_1\sin\theta E_4-a_4\sin\theta
E_5+a_3\sin\theta E_6,\nonumber\\
PE_3&= a_1\sin\theta E_1+ a_2\sin\theta E_2+(a_3^2-a_4^2+(a_2^2-a_1^2)\cos\theta)
E_3\nonumber\\
   &+2(a_3a_4-a_1a_2\cos\theta) E_4 -(a_1a_3+a_2a_4)(1+\cos\theta)E_5\nonumber\\
   &+(a_2a_3-a_1a_4)(-1+\cos\theta)E_6\nonumber\\
PE_4&=a_2\sin\theta E_1-a_1\sin\theta E_2+2(a_3a_4-a_1a_2\cos\theta)E_3\nonumber\\
   &+(a_4^2-a_3^2+(a_1^2-a_2^2)\cos\theta)E_4-(a_2a_3-a_1a_4)(-1+\cos\theta)E_5\nonumber\\
   &-(a_1a_3+a_2a_4)(1+\cos\theta)E_6,\nonumber\\
PE_5&=a_3\sin\theta E_1-a_4\sin\theta E_2-(a_1a_3+a_2a_4)(1+\cos\theta)E_3\nonumber\\
&-(a_2a_3-a_1a_4)(-1+\cos\theta)E_4+(a_1^2-a_2^2+(a_4^2-a_3^2)\cos\theta)E_5\nonumber\\
&+2(a_1a_2-a_3a_4\cos\theta)E_6,\nonumber\\
PE_6&=a_4\sin\theta E_1+a_3\sin\theta
E_2+(a_2a_3-a_1a_4)(-1+\cos\theta)E_3\nonumber\\
&-(a_1a_3+a_2a_4)(1+\cos\theta)E_4+2(a_1a_2-a_3a_4\cos\theta) E_5\nonumber\\
&+(a_2^2-a_1^2+(a_3^2-a_4^2)\cos\theta)E_6,\label{Peovi}
\end{align}
for some differentiable functions $\theta, a_1, a_2, a_3, a_4$ such that $\sum
a_i^2=1$.
\end{lemma}
Note that $E_1$ is determined up to sign if and only if $\cos \theta \ne 0$. If however $\cos \theta$ vanishes on an open set, we regain the freedom of rotation in  $\D_1$.

If we now look at $\mathcal{F}_i(f)$, then it is clear that this is again a CR immersion.
\begin{lemma}\label{Lemma4}
  If we denote the corresponding variables $\theta, a_1, a_2, a_3, a_4$ respectively by $\widehat\theta, \widehat a_1,\widehat a_2,\widehat a_3,\widehat a_4$ and $\widetilde\theta,\widetilde a_1,\widetilde a_2,\widetilde a_3,\widetilde a_4$, for $\mathcal F_1$ and $\mathcal F_2$, we find that:
  \begin{align*} & \widehat{\theta}= \theta, \hspace{3mm}\widehat a_1=a_1,\hspace{3mm}\widehat a_2=-a_2,\hspace{3mm}\widehat a_3=-a_3,\hspace{3mm}\widehat a_4=a_4;\\
  & \widetilde{\theta}= \theta,\hspace{3mm}\widetilde{a}_1=\frac{1}{2}a_1-\frac{\sqrt{3}}{2}a_2,\hspace{3mm}\widetilde{a}_2=\frac{\sqrt{3}}{2}a_1-\frac{1}{2}a_2,\hspace{3mm}\widetilde{a}_3=\frac{1}{2}a_3-\frac{\sqrt{3}}{2}a_4,\\ &\widetilde{a}_4=\frac{\sqrt{3}}{2}a_3+\frac{1}{2}a_4;
  \end{align*} where $\theta,\mbox{ }\widehat{\theta},\mbox{ }\widetilde{\theta}\in [0,\mbox{ }\frac{\pi}{2}]$.
\end{lemma}
\begin{proof}
First, we consider the immersion $\mathcal F_1$. The first vector field $\widehat{E}_1$$\in \D_1$ of the corresponding frame (\ref{Peovi}), is such that the function $\cos\widehat{\theta}=g(P\widehat{E}_1,\widehat{E}_1)$ attains the maximum for $\widehat{E}_1$. Recall from \cite{MV} that $P\circ\mathcal F_1=\mathcal F_1\circ P$. Then it follows straightforwardly that
 \begin{align*}
                              &\cos\widehat{\theta} =\\
                              &=g(P(\cos\alpha\mbox{ } d\mathcal{F}_1(E_1)+\sin\alpha\mbox{ } d\mathcal{F}_1(E_1)),\cos\alpha\mbox{ } d\mathcal{F}_1(E_1)+\sin\alpha\mbox{ } d\mathcal{F}_1(E_1))) \\
                              & =\cos(\theta)\cos(2\alpha).
                            \end{align*} So, the maximum is attained for $\alpha=0$, which implies  $\widehat{\theta}=\theta$. Using relations $$d\mathcal{F}_1(G(X,Y))=-G(d\mathcal{F}_1(X),d\mathcal{F}_1(Y)), \hspace{1cm}d\mathcal{F}_1\circ J=-J\circ d\mathcal{F}_1,$$ we obtain:
                            \begin{align*}
                              & \widehat{E}_1=d\mathcal{F}_1(E_1),  &\widehat{E}_2=J\widehat{E}_1=-d\mathcal{F}_1(E_2),\hspace{1.2cm}\\
                              & \widehat{E}_3=d\mathcal{F}_1(E_3),
                              &\widehat{E}_4=J\widehat{E}_3=-d\mathcal{F}_1(E_4),\hspace{1.2cm}\\
                              &\widehat{E}_5=\sqrt{3}G(\widehat{E}_1,\widehat{E}_3)=-d\mathcal{F}_1(E_5), &\widehat{E}_6=\sqrt{3}G(\widehat{E}_1,\widehat{E}_3)=d\mathcal{F}_1(E_6).
                            \end{align*} Now, straightforwardly we get \begin{align*}
                                                                       \widehat{a}_1 & =\langle P\widehat{E}_1, \widehat{E}_3 \rangle/\sin \widehat{\theta}= a_1.
                                                                     \end{align*} In a similar way we get the expressions for other $\widehat{a}_i$.\\
Look now at the immersion $\mathcal{F}_2$. Again, recall from \cite{MV} that $P\circ d\mathcal{F}_2=d\mathcal{F}_2\circ (-\frac{1}{2}P+\frac{\sqrt{3}}{2}JP)$. We have: \begin{align*}
                              &\cos\widetilde{\theta} =\\
                              &=g(P(\cos\alpha\mbox{ } d\mathcal{F}_2(E_1)+\sin\alpha\mbox{ } d\mathcal{F}_2(E_1)),\cos\alpha\mbox{ } d\mathcal{F}_2(E_1)+\sin\alpha\mbox{ } d\mathcal{F}_2(E_1))) \\
                              & =\cos\theta\mbox{ } \sin(2\alpha-\frac{\pi}{6}).
                            \end{align*} In this case the maximum is attained for $\alpha=\frac{\pi}{3}$, so we can write \\ $$\widetilde{E}_1=\frac{1}{2}d\mathcal{F}_2(E_1)+\frac{\sqrt{3}}{2}d\mathcal{F}_2(E_2)$$ and straightforward computations give us:\begin{align*}
                                                                       &\widetilde{E}_2=\frac{\sqrt{3}}{2}d\mathcal{F}_2(E_1)-\frac{1}{2}d\mathcal{F}_2(E_2), \\
                                                                       &\widetilde{E}_3=d\mathcal{F}_2(E_3),\\
                                                                       &\widetilde{E}_4=-d\mathcal{F}_2(E_4),\\
                                                                       &\widetilde{E}_5=-\frac{1}{2}d\mathcal{F}_2(E_5)-\frac{\sqrt{3}}{2}d\mathcal{F}_2(E_6), \\
         &\widetilde{E}_6=\frac{\sqrt{3}}{2}d\mathcal{F}_2(E_5)-\frac{1}{2}d\mathcal{F}_2(E_6).
                                                                     \end{align*} We have:
\begin{align*}
                                                                                             \widetilde{a}_1 & = \langle P\widetilde{E}_1,\widetilde{E}_3 \rangle/\sin\widetilde\theta\\
                                                                                             &=\langle P\big(\frac{1}{2}d\mathcal{F}_2(E_1)+\frac{\sqrt{3}}{2}d\mathcal{F}_2(E_2)\big), d\mathcal{F}_2(E_3) \rangle/\sin\theta  \\
                                                                                             &=\langle \frac{1}{2} d\mathcal{F}_2\big(-\frac{1}{2}PE_1+\frac{\sqrt{3}}{2}JPE_1\big) +\frac{\sqrt{3}}{2} d\mathcal{F}_2\big(-\frac{1}{2}PE_2+\frac{\sqrt{3}}{2}JPE_2\big), \\ &\hspace{0.6cm}d\mathcal{F}_2(E_3) \rangle/\sin\theta  \\
                                                                                             &=\langle \frac{1}{2}\big(-\frac{1}{2}PE_1+\frac{\sqrt{3}}{2}JPE_1\big) +\frac{\sqrt{3}}{2} \big(-\frac{1}{2}PE_2+\frac{\sqrt{3}}{2}JPE_2\big),E_3 \rangle/\sin\theta   \\
                                                                                             & =\frac{1}{2} a_1-\frac{\sqrt{3}}{2}a_2.
                                                                                           \end{align*} In a similar way we get other $\widetilde {a}_i$.
\end{proof}

\begin{remark}
If $\theta=0$, then $P$ is determined by \begin{align*}
                                             \omega_1 & =a_3^2-a_4^2+a_2^2-a_1^2, \\
                                             \omega_2 & =2(a_3a_4-a_1a_2), \\
                                             \omega_3 & =2(a_1a_3+a_2a_4),
                                           \end{align*} which change as: \begin{align*}
                                                                           &\widehat{\omega}_1=\langle P\widehat{E}_3,\widehat{E}_3\rangle =\omega_1 &\widetilde{\omega}_1=\langle P\widetilde{E}_3,\widetilde{E}_3\rangle=-\frac{1}{2}\omega_1-\frac{\sqrt{3}}{2}\omega_2\\
                                                                           &\widehat{\omega}_2=\langle P\widehat{E}_3,\widehat{E}_4\rangle=-\omega_2 &\widetilde{\omega}_2=\langle P\widetilde{E}_3,\widetilde{E}_4\rangle=-\frac{\sqrt{3}}{2}\omega_1+\frac{1}{2}\omega_2\\
                                                                           &\widehat{\omega}_3=-\langle P\widehat{E}_3,\widehat{E}_5\rangle=-\omega_3 &\widetilde{\omega}_3=-\langle P\widetilde{E}_3,\widetilde{E}_5\rangle=\omega_3.\hspace{50pt}\\
                                                                         \end{align*}
\end{remark}

\section{Case $P\D_1=\D_1$}

Here, we assume that $P\D_1=\D_1$, i.e. $\theta=0$. Recall that $\D_2=Span\{E_3, E_4\}$ and $\D_3=Span\{E_5, E_6\}$. So, in this section we will consider some cases with respect to the position of $P\D_2$. Note that we still have a freedom of rotating the frame of $\D_1$.
\begin{theorem}
  There is no CR submanifold such that $P\D_1=\D_1$ and $P\D_2=\D_3$.
\end{theorem}
\begin{proof}
If we assume that $P\D_1=\D_1$ and $PD_2=D_3$, without loss of generality we can choose $E_1$ such that $PE_1=E_1$ and we know that for $E_3$ we can write  $PE_3=\cos t E_5 + \sin  t E_6$.
Taking $X\in \{E_1, E_2, E_3\}$ and $Y=E_1$ in the second formula in the last line of (\ref{popste}) we get $\Gamma_{11}^2=\Gamma_{21}^2=\Gamma_{31}^2=0$. Similarly, taking $(X,Y)=(E_2,E_1)$ and $(X,Y)=(E_1,E_2)$ and taking the sum, we get $h_{12}^1=0$ and $h_{22}^1=h_{11}^1$. Looking now again when $(X,Y)=(E_1,E_1)$, we get $h_{11}^3=h_{11}^1 \cos t$ and $h_{12}^3=h_{11}^1 \sin t$. From the same equation we also derive $h_{11}^1 \sin t=0$. Now, we consider two cases. First, if we suppose that $\sin t=0$, for the combinations $(X,Y)=(E_1,E_3)$ and $(X,Y)=(E_2,E_1)$ we derive $h_{13}^3=-\frac{1}{2\sqrt{3}}$ and $h_{13}^3=\frac{\sqrt{3}}{2}$. This yields a contradiction. Now we suppose that $h_{11}^1=0$. If $\cos t\neq 0$, we get the same contradiction as in the previous case. If $\cos t=0$, taking $(X,Y)=(E_3,E_1)$ we deduce $h_{13}^1=h_{13}^2=0$ and taking $(X,Y)=(E_3,E_2)$  we find that  $h_{23}^1=0$ and $h_{13}^3=\frac{1}{2\sqrt{3}}$. Computing now $0=\widetilde{\nabla}_{E_1}\widetilde{\nabla}_{E_2}E_3
-\widetilde{\nabla}_{E_2}\widetilde{\nabla}_{E_1}E_3-\widetilde{\nabla}_{[E_1,
E_2]}E_3-R(E_1, E_2)E_3$ we get a contradiction. So, the case that $P\D_1=\D_1$ and $P\D_2=\D_3$ is impossible.
\end{proof}
\begin{remark}
As $P$ is injective and all distributions have the same dimension, the equality  signs in the above theorem can be replaced with inclusions. The same is also true in subsequent theorems.
\end{remark}
\begin{theorem}\label{PD1=D1PD2=D2}
Let $M$ be a $3$-dimensional CR submanifold of $\S^3\times\S^3$ such that $\D_1$ holds $P\D_1=\D_1$ and $P\D_2=\D_2$. Then $M$ is locally congruent to one of following three immersions:
\begin{align*}
 f_1(u,v,t)&=(\cos(\frac{\sqrt{3}}{2}u+\frac{1}{2}v)+i\sin(\frac{\sqrt{3}}{2}u+\frac{1}{2}v),\\
 &A(t)(\cos(\frac{\sqrt{3}}{2}u-\frac{1}{2}v)+i\sin(\frac{\sqrt{3}}{2}u-\frac{1}{2}v))),\\
  f_2(u,v,t)&=(A(t)(\cos(\frac{\sqrt{3}}{2}u-\frac{1}{2}v)+i\sin(\frac{\sqrt{3}}{2}u-\frac{1}{2}v)),\\
  &\cos(\frac{\sqrt{3}}{2}u+\frac{1}{2}v)+i\sin(\frac{\sqrt{3}}{2}u+\frac{1}{2}v)),\\
&\hspace{-17mm}f_3(u,v,t)=((\cos(v)+i\sin(v))\bar{A},(\cos(\frac{\sqrt{3}}{2}u-\frac{1}{2}v)-i\sin(\frac{\sqrt{3}}{2}u-\frac{1}{2}v))\bar{A}),
\end{align*}
where
 $A(t)=a_1(t)+a_2(t)j$ where $a_1(t),a_2(t)\in \C$ and
 \begin{align*}
   a_1'(t) & =-\frac{\sqrt{3}}{2}a_2(t)e^{-if(t)} \\
   a_2'(t) & =\frac{\sqrt{3}}{2}a_1(t)e^{if(t)}
 \end{align*}
for some differentiable function $f$.
\end{theorem}

\begin{proof}
  We suppose that $P\D_1=\D_1$ and $P\D_2=\D_2$, so we can choose $E_1\in \D_1$ such that $PE_1=E_1$. So, $\theta=0$ and we write $\omega_1=-\cos(2\phi)$, $\omega_2=-\sin(2\phi)$ and $\omega_3=0$. For this values almost product structure $P$ is defined on following way:
  \begin{align*}
  & PE_1=E_1,\hspace{3mm}PE_2=-E_2,\hspace{3mm}PE_3=-\cos(2\phi)E_3-\sin(2\phi)E_4,\\
  &PE_4=-\sin(2\phi)E_3+\cos(2\phi)E_4,\hspace{3mm}PE_5=\cos(2\phi)E_5+\sin(2\phi)E_6,\\
  &PE_6=\sin(2\phi)E_5-\cos(2\phi)E_6.
  \end{align*}
Taking combinations $(X,Y)\in \{(E_1,E_1), (E_1,E_2), (E_2,E_1), (E_3,E_1),\\ (E_1,E_3), (E_2,E_3)\}$ in the second equation in the last line of (\ref{popste}) we get
$$\Gamma_{11}^2= \Gamma_{21}^2=\Gamma_{31}^2=0, h_{12}^3=h_{11}^3=0, h_{13}^2=-\frac{\sin(4\phi)}{\sqrt{3}}, h_{23}^2=\frac{2\cos(4\phi)-1}{2\sqrt{3}}.$$ Taking again $(X,Y)= (E_3,E_1)$ in the same equation of (\ref{popste}), we get the following conditions:
\begin{align*}
& \cos(2\phi)+\cos(4\phi)=0,\hspace{10mm} -\sin(2\phi)+\sin(4\phi)=0.
\end{align*}
We can now conclude that the possibilities for $\phi$ are $\{-\frac{\pi}{6}+k\pi,\mbox{  } \frac{\pi}{6}+k\pi,\mbox{  } \frac{\pi}{2}+k\pi\}$, $k\in \R$.

Now, we will consider the case $\phi_1=-\frac{\pi}{6}+k\pi$. If we take $(X,Y)\in \{(E_1,E_3), (E_2,E_3), (E_3,E_3)\}$ in (\ref{popste}) we obtain: $$h_{11}^1=\sqrt{3}h_{12}^1,\mbox{ } h_{22}^1=\frac{1}{\sqrt{3}}h_{12}^1,\mbox{ }  h_{23}^1=0,\mbox{ }  h_{13}^1=0,\mbox{ }  h_{33}^1=0,\mbox{ }  h_{33}^2=\sqrt{3}h_{33}^3.$$
Taking $X=E_1$, $Y=E_2$ and $Z=E_1$ in the Gauss equation we get $h_{12}^1=0$ and if we use again the Gauss equation for $(X,Y,Z)\in\{(E_1,E_3,E_3),(E_2,E_3,E_3)\}$ we get that the derivatives $E_1(h_{33}^3)=E_2(h_{33}^3)\\=0$. Now, straightforward computations give us $$[E_1,E_2]=0,\mbox{   }[E_1,E_3]=0,\mbox{   }[E_2,E_3]=0.$$
So, the vector fields $E_1, E_2, E_3$, i.e. $df(E_1), df(E_2), df(E_3)$ correspond to coordinate vector fields $\partial_u, \partial_v, \partial_t$. The immersion $f: M\rightarrow S^3\times S^3$ can be written as pair of maps $(p,q)$, with $p: M\rightarrow S^3$ and $q: M\rightarrow S^3$. Also, we will use $df(E_i)=(p\alpha_i, q\beta_i)$, where $\alpha_i,$ $\beta_i$ are imaginary quaternion functions.
We can express projections of $df(E_i)$ for $i=1,2,3$ on the tangent space of both spheres using (\ref{Qp}) and we get:\begin{align}\label{pPD1=PD1PD2=PD2}
& p_v=\frac{1}{\sqrt{3}}p_u, \hspace{10mm} p_t=0, \hspace{10mm}q_v=-\frac{1}{\sqrt{3}}q_u.
\end{align}
Now, we obtain the following Euclidean covariant derivatives:
\begin{align}\label{kovarijantniPd1=D1PD2=D2}
&\nabla^E_{\partial_u}dp(E_i)=0\hspace{6mm}\nabla^E_{\partial_v}dp(E_i)=0\hspace{6mm}\nabla^E_{\partial_t}dp(E_i)=0\hspace{6mm}i=1,2,3\nonumber\\
&\nabla^E_{\partial_u}dq(E_1)=0\hspace{6mm}\nabla^E_{\partial_v}dq(E_1)=0\hspace{6mm}\nabla^E_{\partial_t}dq(E_1)= \frac{3}{4}E_5+\frac{\sqrt{3}}{4}E_6 \nonumber  \\
&\nabla^E_{\partial_u}dq(E_2)=0\hspace{6mm}\nabla^E_{\partial_v}dq(E_2)=0\hspace{6mm}\nabla^E_{\partial_t}dq(E_2)= -\frac{\sqrt{3}}{4}E_5-\frac{1}{4}E_6 \nonumber \\
& \nabla^E_{\partial_u}dq(E_3)=\frac{3}{4}E_5+\frac{\sqrt{3}}{4}E_6\hspace{17mm}\nabla^E_{\partial_v}dq(E_3)=-\frac{\sqrt{3}}{4}E_5-\frac{1}{4}E_6\nonumber\\
&\nabla^E_{\partial_t}dq(E_3)=\sqrt{3}h_{33}^3E_5+h_{33}^3E_6.
\end{align}
Note that the coefficient $h_{33}^3$ depends only on the variable $t$.
We can calculate the scalar product using relation (\ref{vezica}) and we have:
\begin{align}\label{duzinePD1=D1PD2=D2}
& \|\alpha_1\|=\|\beta_1\|=\|\beta_3\|=\frac{\sqrt{3}}{2}, \hspace{3mm}\|\alpha_2\|=\|\beta_2\|=\frac{1}{2}, \hspace{3mm}\|\alpha_3\|=0,\nonumber  \\
&\langle\alpha_1,\alpha_2\rangle=\frac{\sqrt{3}}{4},\hspace{3mm} \langle\beta_1,\beta_2\rangle=-\frac{\sqrt{3}}{4},\hspace{3mm} \langle\beta_1,\beta_3\rangle=\langle\beta_2,\beta_3\rangle=0.
\end{align}
So, we obtain $p_{uu}=-\frac{3}{4}p$, this together with (\ref{pPD1=PD1PD2=PD2}) gives us
\begin{align*}
& p(u,v)=A\cos(\frac{\sqrt{3}}{2}u+\frac{1}{2}v)+B\sin(\frac{\sqrt{3}}{2}u+\frac{1}{2}v)
\end{align*}
where $A, B\in \H$ are constants and $\|A\|=\|B\|=1$, $\langle A,B\rangle=0$. By a rotation of the space $\R^4$ we can suppose that the immersion is
\begin{align*}
& p(u,v)=\cos(\frac{\sqrt{3}}{2}u+\frac{1}{2}v)+i\sin(\frac{\sqrt{3}}{2}u+\frac{1}{2}v).
\end{align*}
For $q$ we also have $q_{uu}=-\frac{3}{4}q$ and using (\ref{pPD1=PD1PD2=PD2}) we obtain that $q$ has the form
\begin{align}\label{qPD1=D1PD2=D2}
& q(u,v,t)=\widetilde{A}(t)\cos(\frac{\sqrt{3}}{2}u-\frac{1}{2}v)+ \widetilde{B}(t) \sin(\frac{\sqrt{3}}{2}u-\frac{1}{2}v)
\end{align}
where $\widetilde{A}(t) ,\widetilde{B}(t)\in \H$ and $\|\widetilde{A}(t)\|=\|\widetilde{B}(t)\|=1$, $\langle \widetilde{A}(t),\widetilde{B}(t)\rangle=0$. Straightforward computation gives us
\begin{align}\label{alfePD1=D1PD2=D2}
& \alpha_1=\bar{p}p_u=\frac{\sqrt{3}}{2}i,\hspace{5mm}\alpha_2=\bar{p}p_v=\frac{1}{2}i,\hspace{5mm}\alpha_3=0.
\end{align}
We suppose that $PE_1=E_1$ and $PE_2=-E_2$, i.e. we have
\begin{align}\label{betePD1=D1PD2=D2}
 &\beta_1=\alpha_1=\frac{\sqrt{3}}{2}i,\hspace{5mm}\beta_2=-\alpha_2=-\frac{1}{2}i.
\end{align}
$E_1,$$E_2,$$E_3$ are coordinate vector fields, so using formula $$(0,q_{ij})=-\langle\beta_i,\beta_j\rangle (0,q)+\nabla^E_{E_i}dq(E_j),\mbox{  }i,j\in\{1,2,3\}.$$
from (\ref{kovarijantniPd1=D1PD2=D2}) we get $$\partial_u(\beta_3)=\sqrt{3}i\times \beta_3,\mbox{  }\partial_v(\beta_3)=i\times \beta_3,\mbox{  }\partial_t(\beta_3)=-2h_{33}^3i\times \beta_3.$$
From (\ref{duzinePD1=D1PD2=D2}) we see that $\beta_3\in Im \H$ is orthogonal to $\beta_1$ and $\beta_2$, with length $\frac{\sqrt{3}}{2}$ and we can write $\beta_3=\frac{\sqrt{3}}{2}\cos(\gamma) j+\frac{\sqrt{3}}{2}\sin(\gamma) k$, where $\gamma$ is a differentiable function and a straightforward computation gives us $$\partial_u(\gamma)=-\sqrt{3},\mbox{  }\partial_v(\gamma)=1,\mbox{  }\partial_t(\gamma)=-2h_{33}^3.$$
We can express $\beta_3$ in the following way $$\beta_3=\frac{\sqrt{3}}{2}\cos(-\sqrt{3}u+v+f(t)) j+\frac{\sqrt{3}}{2}\sin(-\sqrt{3}u+v+f(t)) k,$$ where $f(t)$ is differentiable function such that $f'(t)=-2h_{33}^3$.
From $q_t=q\beta_3$ we get conditions for $\widetilde{A}(t)$ and $\widetilde{B}(t)$:
\begin{align*}
& \widetilde{A}'(t)=\frac{\sqrt{3}}{4}\widetilde{A}(t)je^{-if(t)}-\frac{\sqrt{3}}{4}\widetilde{B}(t)ke^{-if(t)}, \\
&\widetilde{B}'(t)=-\frac{\sqrt{3}}{4}\widetilde{A}(t)ke^{-if(t)}-\frac{\sqrt{3}}{4}\widetilde{B}(t)je^{-if(t)}, \\
&0=\frac{\sqrt{3}}{4}\widetilde{A}(t)je^{-if(t)}+\frac{\sqrt{3}}{4}\widetilde{B}(t)ke^{-if(t)}, \\
&0=\frac{\sqrt{3}}{4}\widetilde{A}(t)ke^{-if(t)}-\frac{\sqrt{3}}{4}\widetilde{B}(t)je^{-if(t)}.
\end{align*}
These equations  reduce to:
\begin{align}\label{AiBopetPD1=D1PD2=D2}
&\widetilde{A}'(t)=\frac{\sqrt{3}}{2}\widetilde{A}(t)je^{-if(t)}, \\
&\widetilde{B}(t)= \widetilde{A}(t)i.
\end{align}
If we use the notation $\widetilde{A}(t)=a_1(t)+a_2(t)j$, where $a_1(t),a_2(t)\in \C$, from (\ref{AiBopetPD1=D1PD2=D2}) we obtain the relations:
\begin{align*}
   a_1'(t) & =-\frac{\sqrt{3}}{2}a_2(t)e^{-if(t)} \\
   a_2'(t) & =\frac{\sqrt{3}}{2}a_1(t)e^{if(t)}.
\end{align*} From this moment on, we will use this notation,  $A=\widetilde{A}$ and $B=\widetilde{B}$. Finally, the immersion is
\begin{align*}
  f(u,v,t) & =(\cos(\frac{\sqrt{3}}{2}u+\frac{1}{2}v)+i\sin(\frac{\sqrt{3}}{2}u+\frac{1}{2}v),\\
 &A(t)(\cos(\frac{\sqrt{3}}{2}u-\frac{1}{2}v)+i\sin(\frac{\sqrt{3}}{2}u-\frac{1}{2}v))),
\end{align*} where $ A(t)$ satisfies previous conditions.
Two other cases when $\phi_2=\frac{\pi}{6}+k\pi$ and $\phi_2=\frac{\pi}{2}+k\pi$, $k\in \R$ we can obtain by using angle functions. From $PE_3=-\cos(2\phi)E_3-\sin(2\phi)E_4$ we obtain that $\theta=\phi+\frac{1}{2}\pi$. For $\phi_1=-\frac{\pi}{6}$, $\phi_2=\frac{\pi}{6}$, $\phi_3=\frac{\pi}{2}$ corresponding angles are $\theta_1=\frac{\pi}{3}$, $\theta_2=\frac{2\pi}{3}$, $\theta_3=\pi$. $\theta_2=\pi-\theta_1$ and we can conclude that the immersion $f_2(u,v,t)=\mathcal F_1(p,q)=(q,p)$ corresponds to the angle $\phi_2=\frac{\pi}{6}$. On the other hand, for the immersion \begin{align*}
&f_3(u,v,t)=\mathcal F_1\circ \mathcal F_2\circ\mathcal F_1(p,q)=(p\bar{q},\bar{q})\\
&=((\cos(v)+i\sin(v))\bar{A},\cos(\frac{\sqrt{3}}{2}u-\frac{1}{2}v)-i\sin(\frac{\sqrt{3}}{2}u-\frac{1}{2}v)\bar{A})
\end{align*} corresponding angle is $\theta_3=\pi$. This ends our proof.
\end{proof}

\section{Case $P\D_1\perp \D_1$}

In this section we suppose that $P\D_1\perp \D_1$. First we get a condition when $\D_1$ is not integrable and then we obtain a complete classification of the submanifolds of this type. Moreover, we get an example of a submanifold when the distribution $\D_1$  is integrable and $P\D_1=\D_2$.

\begin{theorem}
   If we suppose that $P\D_1\perp \D_1$, then $\D_1$ is not integrable if and only if $P\D_1=\D_3$.
\end{theorem}\label{PD1=D3}
\begin{proof}
Since $P\D_1\perp \D_1$, we have that $\cos \theta=0, \sin \theta=1 $ in (\ref{Peovi}). If we suppose that $\D_1$ is not integrable, from (\ref{lemma1}) we obtain $h_{11}^1\neq -h_{22}^1$.
Taking $X\in \{E_1, E_2\}$ and $Y=E_2$ in the second equation in the last line of (\ref{popste}) we obtain the equations
\begin{align*}
-h_{11}^1 a_1 + h_{12}^1 a_2 + h_{11}^3 a_3 + h_{12}^3 a_4=0\\
h_{22}^1 a_1 + h_{12}^1 a_2 + h_{11}^3 a_3 + h_{12}^3 a_4=0\\
h_{12}^1 a_1 + h_{11}^1 a_2 - h_{12}^3 a_3 + h_{11}^3 a_4=0\\
h_{12}^1 a_1 - h_{22}^1 a_2 - h_{12}^3 a_3 + h_{11}^3 a_4=0.
\end{align*}
The first two equations reduce to $(h_{11}^1 + h_{22}^1)a_1=0$, so $a_1=0.$ Analo-\\gously, from the third and fourth equation we get $(h_{11}^1 + h_{22}^1)a_2=0$, so $a_2=0$. Hence we can conclude from (\ref{Peovi}) that $P\D_1=\D_3$.\newline
Conversely if we  suppose that $P\D_1=\D_3$ we have $a_1=a_2=0$ in (\ref{Peovi}), so we can write $a_3=\cos t$ and $a_4=\sin t$. In order to obtain a contradiction we can assume that $\D_1$ is integrable, so we have $h_{22}^1= - h_{11}^1$. Taking $X=Y=E_1$ in the second equation in last line of (\ref{popste}) we obtain $ h_{11}^3 \cos t +  h_{12}^3 \sin t=0 \mbox{ and }  h_{12}^3\cos t -  h_{11}^3 \sin t=0$, so, $ h_{11}^3= h_{12}^3=0$. From computing $\widetilde{R}(E_1,E_2)E_6$ in two different ways, we get $h_{12}^1=h_{11}^1=0$.
If we use this in the same equation of (\ref{popste}) when $X=Y=E_1$, we obtain the equations
$$(\sqrt{3}+ 2h_{23}^2) \cos t -2  h_{13}^2 \sin t=0 \mbox{ and }  2h_{13}^2\cos t +  (\sqrt{3}+ 2h_{23}^2) \sin t=0,$$
so, $ h_{13}^2=0$ and $h_{23}^2=-\frac{\sqrt{3}}{2}$. Now taking $(X,Y)=(E_2,E_1)$ in (\ref{popste}) we get that $\cos t$ and $\sin t$ are equal $0$ at the same time, which is impossible.
\end{proof}

\begin{theorem}\label{PD1=D2}
Let $M$ be a $3$-dimensional CR submanifold of $\S^3\times\S^3$ such that $P\D_1=\D_2$. Then $M$ is locally congruent to immersion
\begin{align*}
  f(u,v,t) & =(A(\cos(\sqrt{\frac{3}{2}}u-\frac{1}{\sqrt{2}}v)+\sin\big(\sqrt{\frac{3}{2}}u-\frac{1}{\sqrt{2}}v\big)i\\
   &\hspace{-10mm}+\sqrt{3}\sin\big(\sqrt{\frac{2}{3}}t+\frac{1}{\sqrt{6}}u+\frac{1}{\sqrt{2}}v\big)+\sqrt{3}\cos\big(\sqrt{\frac{2}{3}}t+\frac{1}{\sqrt{6}}u+\frac{1}{\sqrt{2}}v\big)k),\\
   &E(\cos\big(\sqrt{\frac{3}{2}}u+\frac{1}{\sqrt{2}}v\big)+\sin\big(\sqrt{\frac{3}{2}}u+\frac{1}{\sqrt{2}}v\big)i\\
   &\hspace{-11mm}-\sqrt{3}\sin\big(\sqrt{\frac{2}{3}}t+\frac{1}{\sqrt{6}}u-\frac{1}{\sqrt{2}}v\big)j-\sqrt{3}\cos\big(\sqrt{\frac{2}{3}}t+\frac{1}{\sqrt{6}}u-\frac{1}{\sqrt{2}}v\big)k))
\end{align*}
where $A,E\in\H$ and $\|A\|=\|E\|=\frac{1}{2}$.
\end{theorem}

\begin{proof}
 We assume that  $P\D_1=\D_2$ and without loss of generality we can choose $E_1$ such that $PE_1=E_3$. So in expression for the almost product structure $P$ we can take $\theta=\pi/2$ and $a_1=1$, $a_2=a_3=a_4=0$. Also, by using Proposition (\ref{PD1=D3}) we know that the almost complex distribution $\D_1$ is integrable, and (\ref{lemma1}) we have $h_{22}^1=-h_{11}^1$.
 Taking $(X,Y)\in \{(E_1,E_1),(E_2,E_2)\}$ in the second equation in last line of (\ref{popste}) we obtain \begin{align*}
 & h_{11}^1= 0,\hspace{4mm}h_{12}^1= 0,\hspace{4mm}h_{12}^3= 0,\hspace{4mm}h_{13}^2= 0, \\
 & h_{11}^3= -\frac{1}{\sqrt{3}},\hspace{4mm}h_{23}^2= -\frac{1}{2\sqrt{3}},\hspace{4mm}\Gamma_{11}^2= -h_{13}^1,\hspace{4mm}\Gamma_{21}^2= -h_{23}^1
\end{align*} and if we take $(E_3,E_2)$ in (\ref{popste}) we get \begin{align*}
 &  h_{13}^1= 0,\hspace{4mm}h_{23}^1= 0,\hspace{4mm}h_{33}^2= 0,\hspace{4mm}h_{33}^3= 0,\hspace{4mm}\Gamma_{31}^2=-h_{33}^1.
\end{align*} By computing $\widetilde{R}(E_1,E_3)E_1$ in two different ways we get $h_{33}^1=0$.
It now follows that \begin{align*}
 &  [E_1,E_2]= 0,\hspace{4mm}[E_1,E_3]= 0,\hspace{4mm}[E_2,E_3]= 0,
\end{align*} and we obtain that the vector fields $E_1, E_2, E_3$ i.e. $df(E_1), df(E_2), df(E_3)$ correspond to coordinate vector fields $\partial_u, \partial_v, \partial_t$. We write the immersion as $f=(p,q)$ and we can also write $df(E_i)=(p\alpha_i, q\beta_i)$, where $\alpha_i,$ $\beta_i$ are imaginary quaternionic  functions.
Using (\ref{Qp}) we can express $(dp(E_i),0)$ and $(0,dq(E_i))$ and then  straightforward computation gives us the following Euclidean
covariant derivatives: \begin{align}\label{euklidkovexN}
& \nabla^E_{\partial_u}dp(\partial_u)=\frac{1}{3}E_5-\frac{1}{\sqrt{3}}E_6\hspace{16mm}
\nabla^E_{\partial_u}dp(\partial_v)=-\frac{1}{2\sqrt{3}}E_5+\frac{1}{2}E_6\nonumber \\
& \nabla^E_{\partial_v}dp(\partial_u)=-\frac{1}{2\sqrt{3}}E_5+\frac{1}{2}E_6\hspace{11mm}
\nabla^E_{\partial_v}dp(\partial_v)=0\nonumber \\
& \nabla^E_{\partial_t}dp(\partial_u)=-\frac{1}{12}E_5+\frac{1}{4\sqrt{3}}E_6\hspace{10mm}
\nabla^E_{\partial_t}dp(\partial_v)=-\frac{1}{4\sqrt{3}}E_5+\frac{1}{4}E_6\nonumber \\
& \nabla^E_{\partial_u}dp(\partial_t)=-\frac{1}{12}E_5+\frac{1}{4\sqrt{3}}E_6\hspace{10mm}
\nabla^E_{\partial_u}dq(\partial_u)=-\frac{1}{3}E_5-\frac{1}{\sqrt{3}}E_6\nonumber \\
& \nabla^E_{\partial_v}dp(\partial_t)=-\frac{1}{4\sqrt{3}}E_5+\frac{1}{4}E_6\hspace{13mm}
\nabla^E_{\partial_v}dq(\partial_u)=-\frac{1}{2\sqrt{3}}E_5-\frac{1}{2}E_6\nonumber \\
& \nabla^E_{\partial_t}dp(\partial_t)=-\frac{1}{6}E_5+\frac{1}{2\sqrt{3}}E_6\hspace{13mm}
\nabla^E_{\partial_t}dq(\partial_u)=\frac{1}{12}E_5+\frac{1}{4\sqrt{3}}E_6\nonumber \\
& \nabla^E_{\partial_u}dq(\partial_v)=-\frac{1}{2\sqrt{3}}E_5-\frac{1}{2}E_6\hspace{12mm}
\nabla^E_{\partial_u}dq(\partial_t)=\frac{1}{12}E_5+\frac{1}{4\sqrt{3}}E_6\nonumber \\
& \nabla^E_{\partial_v}dq(\partial_v)=0\hspace{39mm}
\nabla^E_{\partial_v}dq(\partial_t)=-\frac{1}{4\sqrt{3}}E_5-\frac{1}{4}E_6\nonumber \\
& \nabla^E_{\partial_t}dq(\partial_v)=-\frac{1}{4\sqrt{3}}E_5-\frac{1}{4}E_6\hspace{12mm}
\nabla^E_{\partial_t}dq(\partial_t)=\frac{1}{6}E_5+\frac{1}{2\sqrt{3}}E_6.
\end{align}
Also, by using (\ref{vezica}) we obtain
\begin{align}\label{metexN}
&\langle\alpha_1,\alpha_2\rangle=\langle\beta_1,\beta_2\rangle=0,\hspace{3mm}\langle\alpha_1,\alpha_3\rangle=\langle\beta_1,\beta_3\rangle=\frac{1}{4},\nonumber\\
&\langle\alpha_2,\alpha_3\rangle=-\langle\beta_2,\beta_3\rangle=\frac{\sqrt{3}}{4},\nonumber\\
&\|\alpha_1\|=\|\alpha_2\|=\|\alpha_3\|=\|\beta_1\|=\|\beta_2\|=\|\beta_3\|=\frac{1}{\sqrt{2}}  .
\end{align}
As we have chosen $E_1$ such that $PE_1=E_3$, we can express all imaginary quaternions with $\alpha_1$ and $\alpha_3$
\begin{align}
& E_1=(p\alpha_1,q\alpha_3),\hspace{1.5mm} E_2=\frac{1}{\sqrt{3}}(p(2\alpha_3-\alpha_1),q(\alpha_3-2\alpha_1)),\hspace{1.5mm} E_3=(p\alpha_3,q\alpha_1),\nonumber \\
&E_4=\frac{1}{\sqrt{3}}(p(2\alpha_1-\alpha_3),q(\alpha_1-2\alpha_3)),\hspace{2mm}
E_5=2(p(\alpha_1\times\alpha_3),q(\alpha_1\times\alpha_3)),\nonumber\\
&E_6=\frac{2}{\sqrt{3}}(p(-\alpha_1\times\alpha_3),q(\alpha_1\times\alpha_3)).\nonumber
\end{align} Using formula $$(p_{ij},0)=-\langle\alpha_i,\alpha_j\rangle (p,0)+\nabla^E_{E_i}dp(E_j),\mbox{  }i,j\in\{1,2,3\}$$
from (\ref{euklidkovexN}) we obtain the derivatives of $\alpha_i$. So, we have \begin{align}\label{alfeexN}
&\partial_u(\alpha_1)=\frac{4}{3}\alpha_1\times\alpha_3\hspace{12mm}\partial_v(\alpha_1)=0\hspace{5mm}\partial_t(\alpha_1)=\frac{2}{3}\alpha_1\times\alpha_3\nonumber \\
&\partial_u(\alpha_2)=-\frac{4}{\sqrt{3}}\alpha_1\times\alpha_3\hspace{5mm}\partial_v(\alpha_2)=0\hspace{5mm}\partial_t(\alpha_2)=-\frac{2}{\sqrt{3}}\alpha_1\times\alpha_3\nonumber\\
& \partial_u(\alpha_3)=-\frac{4}{3}\alpha_1\times\alpha_3\hspace{8mm}\partial_v(\alpha_3)=0\hspace{5mm}\partial_t(\alpha_3)=-\frac{2}{3}\alpha_1\times\alpha_3.
\end{align} From the previous equations we obtain \begin{align}\label{izvexN}
&p\partial_u(\alpha_1)=-p\partial_u(\alpha_3)=p_{uu}+\frac{1}{2}p \\
& p\partial_t(\alpha_3)=-p\partial_t(\alpha_1)=p_{tt}+\frac{1}{2}p.
\end{align} Using the Euclidean covariant derivatives (\ref{euklidkovexN}) and the scalar products (\ref{metexN}) we get the next relations \begin{align}\label{relacije1}
&p_{uu}+\frac{1}{2}p+\frac{2}{\sqrt{3}}p_{vu}=0,\hspace{5mm}p_{tt}+\frac{1}{2}p-\frac{1}{\sqrt{3}}p_{vu}=0,\nonumber \\
&p_{tt}+\frac{1}{2}p_{uu}+\frac{3}{4}p=0,\hspace{5mm}p_{tu}+\frac{1}{4}p_{uu}+\frac{3}{8}p=0,\hspace{5mm}p_{vv}+\frac{1}{2}p=0,\nonumber \\
&p_{tu}=\frac{1}{2}p_{tt},\hspace{5mm}p_{tv}=\frac{\sqrt{3}}{2}p_{tt},\hspace{5mm} p_{tv}+\frac{\sqrt{3}}{4}p_{uu}+\frac{3\sqrt{3}}{8}p=0.
\end{align} Now, if we differentiate equation $p_u=p\alpha_1$ by $u$ and $p_t=p\alpha_3$ by $t$ using derivatives (\ref{alfeexN}) and relations (\ref{izvexN}) we get the following equations: \begin{align}\label{relacije2}
&p_{uuuu}+\frac{5}{3}p_{uu}+\frac{1}{4}p=0,\hspace{5mm}p_{tt}+\frac{2}{3}p_{t}=0,\hspace{5mm}p_{uuu}+\frac{3}{2}p_{u}=\frac{2}{3}p_t.
\end{align} From the first two  equations in (\ref{relacije2}) and $p_{vv}+\frac{1}{2}p=0$ we get that the general solution of $p$ has the form
\begin{align}\label{gensolp}
p(u,v,t)&=((A^0_1+A^1_1\cos(\sqrt{\frac{2}{3}}t)+A^2_1\sin(\sqrt{\frac{2}{3}}t))\cos(\sqrt{\frac{3}{2}}u)\nonumber\\
&+(A^0_2+A^1_2\cos(\sqrt{\frac{2}{3}}t)+A^2_2\sin(\sqrt{\frac{2}{3}}t))\sin(\sqrt{\frac{3}{2}}u) \nonumber\\
&+(A^0_3+A^1_3\cos(\sqrt{\frac{2}{3}}t)+A^2_3\sin(\sqrt{\frac{2}{3}}t))\cos(\frac{1}{\sqrt{6}}u)\nonumber\\
&+(A^0_4+A^1_4\cos(\sqrt{\frac{2}{3}}t)+A^2_4\sin(\sqrt{\frac{2}{3}}t))\sin(\frac{1}{\sqrt{6}}u)))\cos(\frac{1}{\sqrt{2}}v)\nonumber\\
&+((B^0_1+B^1_1\cos(\sqrt{\frac{2}{3}}t)+B^2_1\sin(\sqrt{\frac{2}{3}}t))\cos(\sqrt{\frac{3}{2}}u)\nonumber\\
&+(B^0_2+B^1_2\cos(\sqrt{\frac{2}{3}}t)+B^2_2\sin(\sqrt{\frac{2}{3}}t))\sin(\sqrt{\frac{3}{2}}u) \nonumber\\
& +(B^0_3+B^1_3\cos(\sqrt{\frac{2}{3}}t)+B^2_3\sin(\sqrt{\frac{2}{3}}t))\cos(\frac{1}{\sqrt{6}}u)\nonumber\\
&+(B^0_4+B^1_4\cos(\sqrt{\frac{2}{3}}t)+B^2_4\sin(\sqrt{\frac{2}{3}}t))\sin(\frac{1}{\sqrt{6}}u)))\sin(\frac{1}{\sqrt{2}}v)
\end{align}
for some constants $A^j_i,\mbox{  }B^j_i\in \H$. When we use equation $p_{tv}=\frac{\sqrt{3}}{2}p_{tt}$ we get that the coefficients of $p$ satisfy
\begin{align*}
&B^1_1=A^2_1,\hspace{5mm}B^2_1=-A^1_1,\hspace{5mm}B^1_2=A^2_2,\hspace{5mm}B^2_2=-A^1_2,\nonumber\\
&B^1_3=A^2_3,\hspace{5mm}B^2_3=-A^1_3,\hspace{5mm} B^1_4=A^2_4,\hspace{5mm}B^2_4=-A^1_4.
\end{align*}
Also, when we use equation $p_{tv}=\frac{\sqrt{3}}{2}p_{tt}$ we get the following relations
\begin{align*}
 & A^1_4=A^2_3,\hspace{5mm}A^2_4=-A^1_3,\hspace{5mm}B^1_4=B^2_3,\hspace{5mm}B^2_4=-B^1_3, \\
& A^1_1=A^2_2=A^2_1=A^1_2= B^1_1=B^2_2=B^2_1=B^1_2=0.
\end{align*}
From $p_{uuu}+\frac{3}{2}p_{u}=\frac{2}{3}p_t$ we obtain that $A^0_3=A^0_4=B^0_3=B^0_4=0$  and finally, from $p_{uu}+\frac{1}{2}p+\frac{2}{\sqrt{3}}p_{vu}=0$ we have that $B^0_1=-A^0_2$ and $B^0_2=A^0_1$. From this moment we will use this notation:$A^0_1=A$, $A^0_2=B$, $A^1_3=C$ and $A^2_3=D$. So $p$ has the form
\begin{align*}
p(u,v,t)&=
A\cos(\sqrt{\frac{3}{2}}u-\frac{1}{\sqrt{2}}v)+B\sin(\sqrt{\frac{3}{2}}u-\frac{1}{\sqrt{2}}v)\\ &+C\cos(\sqrt{\frac{2}{3}}t+\frac{1}{\sqrt{6}}u+\frac{1}{\sqrt{2}}v)+D\sin(\sqrt{\frac{2}{3}}t+\frac{1}{\sqrt{6}}u+\frac{1}{\sqrt{2}}v).
\end{align*}
and we can easily check that the other relations in (\ref{relacije1}) and (\ref{relacije2}) are also satisfied.
If we look at (\ref{alfeexN}) we see that sum of $\alpha_1$ and $\alpha_3$ is a constant  imaginary quaternion which length is $\sqrt{\frac{3}{2}}$, so we can write
\begin{align}\label{vezazaalfe}
& \alpha_1+\alpha_3=\sqrt{\frac{3}{2}}I,\hspace{10mm} I\in Im \H ,\hspace{5mm} \langle I,I\rangle =1.
\end{align}
We have $p_u+p_t=\sqrt{\frac{3}{2}}pI$, which implies
\begin{align}\label{vezakoeficijenata}
& B=AI,\hspace{5mm} D=CI\hspace{7mm} and\hspace{7mm} \|A\|=\|B\|,\hspace{5mm} \|C\|=\|D\|.
\end{align}
From these relations we get
\begin{align}\label{Aovi}
& \bar{A}B=-\bar{B}A={\|A\|}^2I\hspace{3mm}and\hspace{3mm}\bar{C}D=-\bar{D}C={\|C\|}^2I.
\end{align}
We can express $\alpha_i$ using $\alpha_1=\bar{p}p_u$, $\alpha_2=\bar{p}p_v$, $\alpha_3=\bar{p}p_t$ and from $\partial_v(\alpha_i)=0$ we obtain
\begin{align}\label{izizvoda}
& \bar{B}C=\bar{A}D,\hspace{3mm}\bar{B}D=-\bar{A}C,\hspace{3mm}\bar{D}B=-\bar{C}A,\hspace{3mm}\bar{C}B=\bar{D}A\hspace{3mm}\bar{B}D=-\bar{A}C.
\end{align}
Using these relations we have
\begin{align*}
\alpha_3 &= \sqrt{\frac{2}{3}}\bar{C}D+\sqrt{\frac{2}{3}}\bar{A}D\cos(\sqrt{\frac{2}{3}}t+\frac{2\sqrt{2}}{\sqrt{3}}u)-\sqrt{\frac{2}{3}}\bar{A}C\sin(\sqrt{\frac{2}{3}}t+\frac{2\sqrt{2}}{\sqrt{3}}u),
\end{align*}
so we can conclude that $\bar{A}D\mbox{,   } \bar{A}C \in Im \H $ and
\begin{align}\label{alfe}
\alpha_1 &=\sqrt{\frac{3}{2}}\bar{A}B+\frac{1}{\sqrt{6}}\bar{C}D-\sqrt{\frac{2}{3}}\bar{A}D\cos(\sqrt{\frac{2}{3}}t+\frac{2\sqrt{2}}{\sqrt{3}}u)\nonumber\\
&+\sqrt{\frac{2}{3}}\bar{A}C\sin(\sqrt{\frac{2}{3}}t+\frac{2\sqrt{2}}{\sqrt{3}}u),\\
\alpha_2 &= -\frac{1}{\sqrt{2}}\bar{A}B+\frac{1}{\sqrt{2}}\bar{C}D+\sqrt{2}\bar{A}D\cos(\sqrt{\frac{2}{3}}t+\frac{2\sqrt{2}}{\sqrt{3}}u)\nonumber\\
&-\sqrt{2}\bar{A}C\sin(\sqrt{\frac{2}{3}}t+\frac{2\sqrt{2}}{\sqrt{3}}u).\nonumber
\end{align}
From (\ref{vezazaalfe}) we get $\bar{A}B+\bar{C}D=I$, which implies ${\|A\|}^2+{\|C\|}^2=1$. Also, from $\partial_u(\alpha_1)=\frac{4}{3}\alpha_1\times\alpha_3 $ we find that $\bar{A}D=\bar{A}C\times I $ and $\bar{A}C=I\times\bar{A}D.$ So we can conclude that $I$, $\bar{A}D$ and $\bar{A}C$ are mutually orthogonal and $\|\bar{A}D\|=\|\bar{A}C\|$.

We can obtain $q$ in the  same way as $p$. Using $$q\partial_u(\alpha_1)=-q\partial_u(\alpha_3)=-q_{uu}-\frac{1}{2}q,\mbox{   } q\partial_t(\alpha_1)=-q\partial_t(\alpha_3)=q_{tt}+\frac{1}{2}q,$$ we get the following relations:
\begin{align}\label{qovi}
&q_{uuuu}+\frac{5}{3}q_{uu}+\frac{1}{4}q=0,\hspace{3mm}q_{tt}+\frac{2}{3}q_{t}=0,\hspace{3mm} p_{vv}+\frac{1}{2}p=0, \nonumber\\
&q_{vu}-\frac{\sqrt{3}}{2}q_{uu}-\frac{\sqrt{3}}{4}q=0,\hspace{3mm}q_{vu}+\sqrt{3}q_{tt}+\frac{\sqrt{3}}{2}q=0,\hspace{3mm}q_{tu}=\frac{1}{2}q_{tt}, \nonumber\\
&q_{tt}+\frac{1}{2}q_{uu}+\frac{3}{4}q=0,\hspace{3mm}q_{tu}+\frac{1}{4}q_{uu}+\frac{3}{8}q=0,\hspace{3mm}q_{tv}=-\frac{\sqrt{3}}{2}q_{tt},  \nonumber\\
& q_{tv}=\frac{\sqrt{3}}{4}q_{uu}+\frac{3\sqrt{3}}{8}q=0,\hspace{3mm}q_{uuu}+\frac{3}{2}q_{u}=\frac{2}{3}q_t,
\end{align}
so, the general solution for $q$ is the same as the general solution for $p$. When we apply this in all previous relations we get that $q$ has the form
\begin{align*}
q(u,v,t)&=
E\cos(\sqrt{\frac{3}{2}}u+\frac{1}{\sqrt{2}}v)+F\sin(\sqrt{\frac{3}{2}}u+\frac{1}{\sqrt{2}}v)\\ &+G\cos(\sqrt{\frac{2}{3}}t+\frac{1}{\sqrt{6}}u-\frac{1}{\sqrt{2}}v)+H\sin(\sqrt{\frac{2}{3}}t+\frac{1}{\sqrt{6}}u-\frac{1}{\sqrt{2}}v),
\end{align*}
for some constants $A,\mbox{ }B ,\mbox{ }C,\mbox{ }D \in \H$.
Also, we have $\alpha_3=\bar{q}q_u$, $\beta_2=\bar{q}q_v$, $\alpha_1=\bar{q}q_t$, so, from (\ref{vezazaalfe}) we get
\begin{align*}
& F=EI,\hspace{5mm} H=GI\hspace{7mm} and\hspace{7mm} \|E\|=\|F\|,\hspace{5mm}\|G\|=\|H\|
\end{align*}
and in the same way as for $p$ we get:
$$\bar{F}G=\bar{E}H,\mbox{ }\bar{F}H=-\bar{E}G,\mbox{ }\bar{H}F=-\bar{G}E,\mbox{ }\bar{G}F=\bar{H}E,\mbox{ }\bar{F}H=-\bar{E}G\mbox{ and}$$
$\bar{E}G, \bar{E}H \in \H. $ So, we obtain:
\begin{align}\label{alfeq}
\alpha_1 &=\sqrt{\frac{2}{3}}\bar{G}H+\sqrt{\frac{2}{3}}\bar{E}H\cos(\sqrt{\frac{2}{3}}t+\frac{2\sqrt{2}}{\sqrt{3}}u)-\sqrt{\frac{2}{3}}\bar{E}G\sin(\sqrt{\frac{2}{3}}t+\frac{2\sqrt{2}}{\sqrt{3}}u),\\
\beta_2 &= \frac{1}{\sqrt{2}}\bar{E}F-\frac{1}{\sqrt{2}}\bar{G}H-\sqrt{2}\bar{E}H\cos(\sqrt{\frac{2}{3}}t+\frac{2\sqrt{2}}{\sqrt{3}}u)\nonumber\\
&+\sqrt{2}\bar{E}G\sin(\sqrt{\frac{2}{3}}t+\frac{2\sqrt{2}}{\sqrt{3}}u),\nonumber\\
\alpha_3 &= \sqrt{\frac{3}{2}}\bar{E}F+\frac{1}{\sqrt{6}}\bar{G}H-\sqrt{\frac{2}{3}}\bar{E}H\cos(\sqrt{\frac{2}{3}}t+\frac{2\sqrt{2}}{\sqrt{3}}u)\nonumber\\
&+\sqrt{\frac{2}{3}}\bar{E}G\sin(\sqrt{\frac{2}{3}}t+\frac{2\sqrt{2}}{\sqrt{3}}u)\nonumber
\end{align}
and $\bar{E}F+\bar{G}H=I, \mbox{ }{\|E\|}^2+{\|G\|}^2=1,\mbox{ }\bar{E}H=\bar{E}G\times I $ and $\bar{E}G=I\times\bar{E}H$.
If we compare (\ref{alfe}) and (\ref{alfeq}) we get the following relations between the coefficients $$\bar{E}G=-\bar{A}C,\mbox{ }\bar{E}H=-\bar{A}D,\mbox{ }\bar{G}H=\frac{3}{2}\bar{A}B+\frac{1}{2}\bar{C}D,\mbox{ }\bar{E}F=-\frac{1}{2}\bar{A}B+\frac{1}{2}\bar{C}D.$$
For the initial condition $\alpha_1(0,0)=\frac{\sqrt{3}}{2\sqrt{2}}i-\frac{\sqrt{3}}{2\sqrt{2}}j$, we have the unique solution \newline $\alpha_1=\frac{\sqrt{3}}{2\sqrt{2}}i-\frac{\sqrt{3}}{2\sqrt{2}}\cos(\sqrt{\frac{2}{3}}t+\frac{2\sqrt{2}}{\sqrt{3}}u)j+\sin(\sqrt{\frac{2}{3}}t+\frac{2\sqrt{2}}{\sqrt{3}}u)k$. We get $\bar{A}D=\frac{\sqrt{3}}{4}j$ and $\bar{A}C=\frac{\sqrt{3}}{4}k$, so $I=i$ and $\bar{C}D=\frac{3}{4}i$ and it immediately follows that $\bar{A}B=\frac{1}{4}i$ and $\|A\|=\|E\|=\frac{1}{2}$.
Hence the immersion is given by \begin{align*}f(u,v,t) & =(A(\cos(\sqrt{\frac{3}{2}}u-\frac{1}{\sqrt{2}}v)+\sin\big(\sqrt{\frac{3}{2}}u-\frac{1}{\sqrt{2}}v\big)i\\
   &\hspace{-10mm}+\sqrt{3}\sin\big(\sqrt{\frac{2}{3}}t+\frac{1}{\sqrt{6}}u+\frac{1}{\sqrt{2}}v\big)+\sqrt{3}\cos\big(\sqrt{\frac{2}{3}}t+\frac{1}{\sqrt{6}}u+\frac{1}{\sqrt{2}}v\big)k),\\
   &E(\cos\big(\sqrt{\frac{3}{2}}u+\frac{1}{\sqrt{2}}v\big)+\sin\big(\sqrt{\frac{3}{2}}u+\frac{1}{\sqrt{2}}v\big)i\\
   &\hspace{-11mm}-\sqrt{3}\sin\big(\sqrt{\frac{2}{3}}t+\frac{1}{\sqrt{6}}u-\frac{1}{\sqrt{2}}v\big)j-\sqrt{3}\cos\big(\sqrt{\frac{2}{3}}t+\frac{1}{\sqrt{6}}u-\frac{1}{\sqrt{2}}v\big)k)).\end{align*}
\end{proof}

\begin{corollary}
Let $M$ be a
$3$-dimensional CR submanifold of $\S^3\times\S^3$ such that $P\D_1=
\D_3$. Then $M$ is locally congruent with one of the following immersions
\begin{align*}
f_1(u)&=(u \frac{\sqrt{3}+i}{2} u^{-1}, u^{-1}),\\
f_2(u)&=(u^{-1}, u \frac{\sqrt{3}+i}{2} u^{-1}),\\
f_3(u)&=(u, u \frac{\sqrt{3}+i}{2}),
\end{align*}
where $u\in \S^3$.
\end{corollary}
\begin{proof}
Here, we assume that the distribution $\D_1$ is not integrable, so we have that $h_{11}^1\neq -h_{22}^1$. Now, from the second equality in the third line of (\ref{popste}) we have that
\begin{align}
0&=g(E_1, G(E_1, PE_1)+P(G(E_1, E_1))+2J((\widetilde{\nabla}_{E_1} P) E_1))\nonumber\\
&=4(-h_{11}^1 a_1+a_3h_{11}^3+a_4 h_{12}^3),\nonumber\\
0&=g(E_2, G(E_1, PE_1)+P(G(E_1, E_1))+2J((\widetilde{\nabla}_{E_1} P) E_1))\nonumber\\
&=-4(h_{12}^1 a_1+a_4 h_{11}^3-a_3 h_{12}^3),\nonumber\\
0&=g(E_1, G(E_2, PE_1)+P(G(E_2, E_1))+2J((\widetilde{\nabla}_{E_2} P) E_1))\nonumber\\
&=-4(h_{22}^1 a_1+a_3 h_{11}^3+a_4 h_{12}^3),\nonumber\\
0&=g(E_1, G(E_3, PE_1)+P(G(E_3, E_1))+2J((\widetilde{\nabla}_{E_3} P) E_1))\nonumber\\
&=-4 h_{13}^1 a_1+\frac{2}{\sqrt{3}}a_3-4 a_4h_{13}^2+4a_3h_{23}^2,\nonumber\\
0&=g(E_2, G(E_3, PE_1)+P(G(E_3, E_1))+2J((\widetilde{\nabla}_{E_1} P) E_1))\nonumber\\
&=-4 h_{23}^1 a_1+\frac{4}{\sqrt{3}}a_3-4
a_4h_{13}^2+4a_3h_{23}^2.\label{imeneko}
\end{align}
Moreover, then for any pair of
vector fields $E_1$, $E_2$ the corresponding functions $a_1$ and $a_2$ vanish, so we
still have a
freedom of choice for the frame of $\D_1$. Also $a_3^2+a_4^2=1$, so the first and the
second equations of (\ref{imeneko}) now imply that $h_{11}^3=h_{12}^3=0$. Similarly,
from the fourth
and the fifth  equations of (\ref{imeneko}) we get that $h_{13}^2=0,$
$h_{23}^2=-\frac{1}{2\sqrt{3}}$. Further on, we denote by $t$ a locally differentiable function, such that
$a_3=\cos t, a_4=\sin t$.
Now, we have
\begin{align}
0&=G(E_1, PE_1)+P(G(E_1, E_1))+2J((\widetilde{\nabla}_{E_1} P) E_1)\nonumber\\
&=2(\frac{1}{\sqrt{3}}\cos t-h_{11}^1 \cos 2t+h_{12}^1\sin2t)E_3\nonumber\\
&-2(\frac{1}{\sqrt{3}}\sin t+h_{12}^1\cos 2t+h_{11}^1 \sin
2t)E_4\nonumber\\
&+ 2\cos t(E_1(t)+h_{13}^1)E_5+2\sin t(E_1(t)+h_{13}^1),\nonumber\\
0&=G(E_2, PE_1)+P(G(E_2, E_1))+2J((\widetilde{\nabla}_{E_2} P) E_1)\nonumber\\
&=2(\frac{1}{\sqrt{3}}\sin t+h_{22}^1\sin 2t-h_{12}^1\cos 2t)E_3\nonumber\\
&+
2(\frac{1}{\sqrt{3}}\cos t-h_{22}^1\cos 2t+h_{12}^1\sin 2t)E_4\nonumber\\
&+2\cos t(E_2(t)+h_{23}^1)E_5+2\sin t(E_2(t)+h_{23}^1)E_6,\nonumber\\
0&=G(E_3, PE_1)+P(G(E_3, E_1))+2J((\widetilde{\nabla}_{E_3} P) E_1)\nonumber\\
&=2(h_{33}^3\cos t-h_{33}^2\sin t-h_{13}^1\cos 2t+h_{23}^1\sin
2t)E_3\nonumber\\
&-2(h_{33}^3\sin t+h_{33}^2\cos t+h_{23}^1\cos 2t+h_{13}^1\sin
2t)E_4\nonumber\\
&+2\cos t(E_3(t)+h_{33}^1)E_5+2\sin t(E_3(t)+h_{33}^1)E_6.\label{kuda}
\end{align}
Straightforwardly, this implies that
\begin{align}
E_1(t)=-h_{13}^1,\quad E_2(t)=-h_{23}^1,\quad
E_3(t)=-h_{33}^1.\label{izvodi}
\end{align}
Moreover, solving the first equation of (\ref{kuda}) for $h_{12}^1,
h_{11}^1$ yields
\begin{align*}
h_{12}^1=-\frac{\sin 3t}{\sqrt{3}},\quad h_{11}^1=\frac{\cos
3t}{\sqrt{3}},
\end{align*}
and the solution of the second equation of (\ref{kuda}) is given by
\begin{align*}
h_{12}^1=\frac{\sin 3t}{\sqrt{3}},\quad h_{22}^1=\frac{\cos 3t}{\sqrt{3}}.
\end{align*}
Hence, we have that $h_{12}^1=\sin 3t/\sqrt{3}=0$ and $t$ is a constant function
of the form $t=k \pi/3, k\in \mathbb Z$. Moreover, (\ref{izvodi}) yields that
\begin{align*}
h_{23}^1=0,\quad h_{13}^1=0,\quad h_{33}^1=0.
\end{align*}
We have
\begin{align*}
0&=G(E_3, PE_1)+P(G(E_3, E_1))+2J((\widetilde{\nabla}_{E_3} P) E_1)\nonumber\\
&=2(h_{33}^3\cos t-h_{33}^2\sin t)E_3-2(h_{33}^3\sin t+h_{33}^2\cos t)E_4,
\end{align*}
implying
\begin{align*}
h_{33}^2=0,\quad h_{33}^3=0.
\end{align*}
Now, we consider the curvature tensor of $\widetilde{\nabla}$. Computations show that
\begin{align}
0&=\widetilde{\nabla}_{E_1}\widetilde{\nabla}_{E_2}E_1
-\widetilde{\nabla}_{E_2}\widetilde{\nabla}_{E_1}E_1-\widetilde{\nabla}_{[E_1,
E_2]}E_1-R(E_1, E_2)E_1\nonumber\\
&=(\frac{4}{3}+(\Gamma_{11}^2)^2+(\Gamma_{21}^2)^2+\frac{2}{\sqrt{3}}\Gamma_{31}^2\cos
3t-E_2(\Gamma_{11}^2)+E_1(\Gamma_{21}^2))E_2,\nonumber\\
0&=\widetilde{\nabla}_{E_1}\widetilde{\nabla}_{E_3}E_1
-\widetilde{\nabla}_{E_3}\widetilde{\nabla}_{E_1}E_1-\widetilde{\nabla}_{[E_1,
E_3]}E_1-R(E_1, E_3)E_1\nonumber\\
&=(\Gamma_{21}^2 \Gamma_{31}^2-\frac{1}{\sqrt{3}}\Gamma_{21}^2\cos
3t-E_3(\Gamma_{11}^2)+E_1(\Gamma_{31}^2))E_2,\nonumber\\
0&=\widetilde{\nabla}_{E_2}
\widetilde{\nabla}_{E_3}E_1
-\widetilde{\nabla}_{E_3}\widetilde{\nabla}_{E_2}E_1-\widetilde{\nabla}_{[E_2,
E_3]}E_1-R(E_2, E_3)E_1\nonumber\\
&=(-\Gamma_{11}^2\Gamma_{31}^2+\frac{1}{\sqrt{3}}\cos
3t-E_3(\Gamma_{21}^2)+E_2(\Gamma_{31}^2))E_2\label{jedkriv}
\end{align}
where $\widetilde{R}(X, Y)Z$ is given by
(\ref{krivina}).
Let us consider an arbitrary solution of the system (\ref{jedkriv}). Recall that
\begin{align}\Gamma_{11}^2&=-g([E_1, E_2], E_1),\nonumber\\
\Gamma_{21}^2&=-g([E_1, E_2], E_2).\label{ime1}
\end{align}
Let $w$ be a differentiable function given by $E_1(w)=-\Gamma_{11}^2,
E_2(w)=-\Gamma_{21}^2,
E_3(w)=-(2+\sqrt{3}\Gamma_{31}^2 \cos 3t)\sec 3t/\sqrt{3}$. Since
\begin{align*}
(\widetilde{\nabla}_{E_1}E_2&-\widetilde{\nabla}_{E_2}E_1-[E_1, E_2])w\nonumber\\
&=\frac{4}{3}+(\Gamma_{11}^2)^2+(\Gamma_{21}^2)^2+\frac{2}{\sqrt{3}}\Gamma_{31}^2\cos
3t-E_2(\Gamma_{11}^2)+E_1(\Gamma_{21}^2)=0,\\
(\widetilde{\nabla}_{E_1}E_3&-\widetilde{\nabla}_{E_3}E_1-[E_1, E_3])w\nonumber\\
&=\Gamma_{21}^2 \Gamma_{31}^2-\frac{1}{\sqrt{3}}\Gamma_{21}^2\cos
3t-E_3(\Gamma_{11}^2)+E_1(\Gamma_{31}^2)=0,\\
(\widetilde{\nabla}_{E_2}E_3&-\widetilde{\nabla}_{E_3}E_2-[E_2, E_3])w\nonumber\\
&=-\Gamma_{11}^2\Gamma_{31}^2+\frac{1}{\sqrt{3}}\Gamma_{11}^2\cos
3t-E_3(\Gamma_{21}^2)+E_2(\Gamma_{31}^2)=0,
\end{align*}
the integrability conditions for $w$ are satisfied and such function exists. Then, we
can again consider a rotation in the distribution $\D_1$ given by
\begin{align*}
\widetilde{E}_1&=\cos w E_1+\sin w E_2,\\
\widetilde{E}_2&=-\sin w E_1+\cos w E_2.
\end{align*}
Straightforward computation then shows that
\begin{align*}
[\widetilde{E}_1, \widetilde{E}_2]=-(\frac{2}{\sqrt{3}}\cos 3t)E_3.
\end{align*}
Therefore, from (\ref{ime1}) we have that
$\widetilde{\Gamma}_{11}^2=\widetilde{\Gamma}_{21}^2=0$ and further, from
(\ref{jedkriv}) follows that $\widetilde{\Gamma}_{31}^2=-2\sec
(3t)/\sqrt{3}=-2\cos (3t)/\sqrt{3}$. From now on, we deal with this particular frame
and omit the  over-tildes. Note that now we have that $\widetilde{\nabla}_{E_i}E_i=0$
and moreover
\begin{align*}
\widetilde{\nabla}_{E_1}E_2&=-\frac{\cos 3t}{\sqrt{3}}E_3,&
\widetilde{\nabla}_{E_1}E_3&=\frac{\cos 3t}{\sqrt{3}}E_2,\\
\widetilde{\nabla}_{E_2}E_1&=\frac{\cos 3t}{\sqrt{3}}E_3,&
\widetilde{\nabla}_{E_2}E_3&=-\frac{\cos 3t}{\sqrt{3}}E_1,\\
\widetilde{\nabla}_{E_3}E_1&=-\frac{2\cos 3t}{3t}E_2,&
\widetilde{\nabla}_{E_3}E_2&=\frac{2\cos 3t}{3t}E_1.
\end{align*}
We recall here the notion of a Berger sphere $(\S^3, g_1)$ where the metric $g_1$ has the expresion
\begin{align*}
g_1(X,Y)=\frac{4}{\kappa}(\langle X, Y\rangle+(\frac{4\tau^2}{\kappa}-1)\langle X,
X_1\rangle\langle Y, X_1\rangle),
\end{align*}
where $X_1=pi,\mbox{ } X_2=pj,\mbox{ } X_3=-pk$ and $\tau$, $\kappa$ are constants.
Then the vector fields $\overline{E}_1=\frac{\kappa}{4\tau}X_1,$
$\overline{E}_2=\frac{\sqrt{\kappa}}{2}X_2,$
$\overline{E}_3=\frac{\sqrt{\kappa}}{2} X_3$ are orthogonal and a unit with respect to the metric
$g_1$. The Levi-Civita connection $\overline{\nabla}$ of the metric $g_1$ is then
given by
$\overline{\nabla}_{\overline{E}_i}\overline{E}_i=0$ and
\begin{align*}
&\overline{\nabla}_{\overline{E}_1}\overline{E}_2
=-\overline{\nabla}_{\overline{E}_2}\overline{E}_1
=(\tau-\frac{\kappa}{2\tau})\overline{E}_3,\\
&\overline{\nabla}_{\overline{E}_2}\overline{E}_3
=-\overline{\nabla}_{\overline{E}_3}\overline{E}_2=-\tau \overline{E}_1,\\
&\overline{\nabla}_{\overline{E}_3}\overline{E}_1=
-\overline{\nabla}_{\overline{E}_1}\overline{E}_3
=-\tau \overline{E}_2.
\end{align*}
Moreover, the following proposition, see \cite{DVV2} shows that a manifold
admitting such vector fields is locally isometric to a Berger sphere.
\begin{proposition} Let $(M_1,\nabla^1)$ and $(M_2,\nabla^2)$ be $n$-dimensional
Riemannian manifolds with Levi-Civita connections. Suppose that there exist constants
$c_{ij}^k$, $i, j,
k\in\{1,\dots, n\}$ such that for any $p_1\in M_1$ and $p_2\in M_2$ there exist
orthonormal frames $E^1_1,\dots, E^1_n$ and $E^2_1,\dots, E^2_n$ around $p_1$ and
$p_2$ respectively,
such that $\nabla^1_{E^1_i}E^1_j=\sum c_{ij}^k E^1_k$ and $\nabla^2_{E^2_i}E^2_j=\sum
c_{ij}^k E^2_k$. Then, for every point $p_1\in M_1$ and $p_2\in M_2$ there exists a
local isometry
mapping some neighbourhood of $p_1$ into some neighbourhood of $p_2$.
\end{proposition}
In our case, by taking $E_1=\frac{\sqrt{2}}{2} X_2,$ $E_2=\frac{\sqrt{2}}{2} X_3,$ $E_3=\frac{\sqrt{3}}{2}\cos
3t X_1$ we see that $M$ is then congruent to a Berger sphere with $\tau=\frac{\cos
3t}{\sqrt{3}}$ and
$\kappa=2$. Further on, the straightforward computation shows that the Euclidean
covariant derivatives are given by
\begin{align}
&\nabla^E_{E_1}E_1=\frac{\sin t}{\sqrt{3}}E_3-\frac{4\cos t\sin^2
t}{\sqrt{3}}E_4,\nonumber\\
&\nabla^E_{E_1}E_2=-\frac{\cos 3t}{\sqrt{3}}E_3,\nonumber\\
&\nabla^E_{E_1}E_3=-\frac{\sin
t}{2\sqrt{3}}E_1+\frac{\cos t+2\cos 3t}{2\sqrt{3}}E_2+\frac{\cos t\sin
t}{\sqrt{3}}E_5+\frac{\sin^2
t}{\sqrt{3}}E_6,\nonumber\\
&\nabla^E_{E_2}E_1=\frac{\cos 3t}{\sqrt{3}}E_3,\nonumber\\
&\nabla^E_{E_2}E_2=\frac{\sin t}{\sqrt{3}}E_3-\frac{4\cos
t\sin^2t}{\sqrt{3}}E_4,\nonumber\\
&\nabla^E_{E_2}E_3=-\frac{\cos t+2\cos 3t}{2\sqrt{3}}E_1-\frac{\sin
t}{2\sqrt{3}}E_2-\frac{\sin^2t}{\sqrt{3}}E_5+\frac{\cos t\sin
t}{\sqrt{3}}E_6,\nonumber\\
&\nabla^E_{E_3}E_1=-\frac{\sin t}{2\sqrt{3}}E_1+\frac{\cos t-4 \cos
3t}{2\sqrt{3}}E_2+\frac{\cos t\sin
t}{\sqrt{3}}E_5+\frac{\sin^2t}{\sqrt{3}}E_6,\nonumber\\
&\nabla^E_{E_3}E_2=-\frac{\cos t-4\cos 3t}{2\sqrt{3}}E_1-\frac{\sin
t}{2\sqrt{3}}E_2-\frac{\sin^2 t}{\sqrt{3}}E_5+\frac{\cos t\sin
t}{\sqrt{3}}E_6,\nonumber\\
&\nabla^E_{E_3}E_3=0.\label{kovarijantni}
\end{align}
We write the immersion as $f=(p, q)$ and $df(E_i)=(p\alpha_i, q\beta_i)$
where $\alpha_i,$ $\beta_i$ are imaginary quaternion functions.
We  have that
\begin{align*}
PE_1&=\cos t E_5+\sin t E_6,\\
PE_3&=(p\beta_3, q\alpha_3)=\cos 2t E_3+\sin 2t E_4\\
&=\cos 2t (p\alpha_3, q\beta_3)+ \frac{\sin 2t}{\sqrt{3}}(p(2\beta_3-\alpha_3),
q(-2\alpha_3+\beta_3))\\
&= \big(p\big(\cos 2t\alpha_3+\frac{\sin 2t(2\beta_3-\alpha_3)}{\sqrt{3}}\big), q\big(\cos
2t\beta_3+\frac{\sin 2t(-2\alpha_3+\beta_3)}{\sqrt{3}}\big)\big),
\end{align*}
which then simplifies to
\begin{align}
\beta_3&=\cos 2t\alpha_3+\frac{\sin 2t}{\sqrt{3}}(2\beta_3-\alpha_3),\nonumber\\
\alpha_3&=\cos 2t \beta_3+\frac{\sin 2t}{\sqrt{3}}(-2\alpha_3+\beta_3).\label{pe3}
\end{align}

{\bf Case $t=\pi/3+k \pi$}\\
Let $t=\pi/3+k\pi$. Then $\cos t=\varepsilon/2$, $\sin t=\varepsilon\sqrt{3}/2$, for
$\varepsilon\in\{1,-1\}$.
From (\ref{pe3}) it follows that $\alpha_3=0$.
Further on
\begin{align*}
E_5&=\frac{2}{3}(p(\alpha_1\times\beta_3-2\beta_1\times\beta_3),q(-\alpha_1\times\beta_3-\beta_1\times\beta_3)),\\
E_6&=-\frac{2}{\sqrt{3}}(p(-\alpha_1\times\beta_3),q(-\alpha_1\times\beta_3+\beta_1\times\beta_3))
\end{align*}
and
\begin{align*}
PE_1&=\cos t E_5+\sin t E_6\\
&=\frac{2}{3}\varepsilon(p(2\alpha_1\times\beta_3-\beta_1\times\beta_3),q(-2\beta_1\times\beta_3+\alpha_1\times\beta_3)).
\end{align*}
Therefore, we have
\begin{align}
\beta_1&=\frac{2}{3}\varepsilon(2\alpha_1\times\beta_3-\beta_1\times\beta_3),\nonumber\\
\alpha_1&=\frac{2}{3}\varepsilon(-2\beta_1\times\beta_3+\alpha_1\times\beta_3).\label{neke}
\end{align}
Since
\begin{align*}
\alpha_1&=\frac{\beta_1-\sqrt{3}\beta_2}{2},\quad
\alpha_2=\frac{\sqrt{3}\beta_1+\beta_2}{2},
\end{align*}
the first equation of (\ref{neke}) simplifies to $\beta_1=2/\sqrt{3}\varepsilon
\beta_2\times\beta_3$ and the second one becomes $\beta_2=2/\sqrt{3}\varepsilon
\beta_1\times\beta_3$. Of course, then $\beta_1$ and $\beta_2$ are orthogonal and $\|\beta_1\|=\|\beta_2\|=1/\sqrt{2}$.
Then
$\beta_3=-\sqrt{3}\varepsilon\beta_1\times\beta_2$.

Therefore, there is a unit quaternion $h$ such that at the point $u=1$ it holds
\begin{align*}
\beta_1(1)&=-\frac{1}{\sqrt{2}}h j h^{-1},\nonumber\\
\beta_2(1)&=\frac{1}{\sqrt{2}}h k h^{-1},\nonumber\\
\beta_3(1)&=\frac{\sqrt{2}}{2}h i h^{-1}.\label{trebace0}
\end{align*}

Now, (\ref{koneksija}) and (\ref{kovarijantni}) yield that

\begin{align}
&E_1(\beta_1)=0,\quad  E_2(\beta_1)=-2\frac{\varepsilon}{\sqrt{3}}\beta_3,\quad
E_3(\beta_1)=\sqrt{3}\beta_2\varepsilon,\nonumber\\
&E_1(\beta_2)=2\frac{\varepsilon}{\sqrt{3}}\beta_3,\quad E_2(\beta_2)=0,\quad
E_3(\beta_2)=-\sqrt{3}\beta_1\varepsilon,\nonumber\\
&E_1(\beta_3)=-\sqrt{3}\beta_2\varepsilon,\quad
E_2(\beta_3)=\sqrt{3}\beta_1\varepsilon,\quad E_3(\beta_3)=0.
\end{align}
This system of differential equations can have a unique solution with prescribed
initial conditions. A straightforward computation show that this solution is given by
\begin{align}
\beta_1&=-\frac{1}{\sqrt{2}}h u j u^{-1} h^{-1},\nonumber\\
\beta_2&=\frac{1}{\sqrt{2}}h u k u^{-1} h^{-1},\nonumber\\
\beta_3&=\frac{\sqrt{2}}{2}h u i u^{-1} h^{-1}.\label{trebace2}
\end{align}
Consequently, we have
\begin{align}
X_1(q)&=- q h u i u^{-1} h^{-1},\nonumber\\
X_2(q)&=- q h u j u^{-1} h^{-1},\nonumber\\
X_3(q)&= q h u k u^{-1} h^{-1}.\label{trebace3}
\end{align}
Indeed, with the initial condition $q(1)=1$, we see that $q(u)=h u^{-1}h$ is the
solution.
In a similar manner we also have that
\begin{align*}
X_1(p)&=0,\\
X_2(p)&=-h u \frac{j+\sqrt{3}k}{2}u^{-1}h^{-1},\\
X_3(p)&=hu\frac{-\sqrt{3}j+k}{2}u^{-1}h^{-1}.
\end{align*}
Let now $g$ be a unit quaternion such that $g h (\sqrt{3}+i)/2 h^{-1}=1$. Then
$p=g h u (\sqrt{3}+i)/2 u^{-1} h^{-1}$ is the solution of the system for the initial
condition $p(1)=1$.
Hence, the immersion is given by
\begin{align*}u\mapsto(g h u \frac{\sqrt{3}+i}{2}
u^{-1} h^{-1}, h u^{-1}h),\end{align*}
and it is clearly congruent to
$f_1: u\mapsto(u (\sqrt{3}+i)/2 u^{-1}, u^{-1})$.

The two other cases when $t_2=\frac{2}{3}\pi+k\pi$ and $t_3=k\pi$, $k\in \R$, can be obtained by using Lemma (\ref{Lemma4}). From $PE_3=\cos 2t E_3+\sin 2t E_4$ we get that $\theta=t$. For $t_1=\frac{1}{3}\pi$, $t_2=\frac{2}{3}\pi$, $t_3=0$ the corresponding angles are $\theta_1=\frac{1}{3}\pi$, $\theta_2=\frac{2}{3}\pi$, $\theta_3=0$. We notice that $\theta_2=\pi-\theta_1$ and we can conclude that the immersion \begin{align*}
f_2(u)=\mathcal F_1(p,q)=(q,p)=(u^{-1}, u \frac{\sqrt{3}+i}{2} u^{-1})
\end{align*} corresponds to the angle $t_2=\frac{2}{3}\pi$. On the other hand, for the immersion \begin{align*}
&f_3(u)=\mathcal F_2\circ\mathcal F_1(p,q)=(\bar{q},p\bar{q})=(u, u \frac{\sqrt{3}+i}{2})
\end{align*} the corresponding angle is $\theta_3=0$.
This ends the proof.

\end{proof}

\section{Acknowledgement}

The authors would like to thank Prof. Luc Vrancken for  meaningful   suggestions  and discussions on the subject of CR submanifolds of $\S^3\times\S^3$.\\
The research of the first and the second author was supported by the Ministry of
Science and Technological Development of the Republic of Serbia, project 174012.
The third author is a Postdoctoraal Onderzoeker van het Fonds Wetenschappelijk Onderzoek-Vlaanderen.

\end{document}